\documentclass[12pt]{amsart}
\usepackage{amssymb}
\usepackage{graphicx}
\usepackage[colorlinks=true,linkcolor=black,citecolor=black]{hyperref}

\newtheorem{theorem}{Theorem}[section]
\newtheorem{definition}[theorem]{Definition}
\newtheorem{proposition}[theorem]{Proposition}
\newtheorem{conjecture}[theorem]{Conjecture}

\begin{document}

\title{Almost Hadamard matrices with complex entries}

\author{Teodor Banica}
\address{T.B.: Department of Mathematics, University of Cergy-Pontoise, F-95000 Cergy-Pontoise, France. {\tt teodor.banica@u-cergy.fr}}

\author{Ion Nechita}
\address{I.N.: Zentrum Mathematik, M5, Technische Universit\"at M\"unchen, Boltzmannstrasse 3, 85748 Garching, Germany and CNRS, Laboratoire de Physique Th\'eorique, IRSAMC, Universit\'e de Toulouse, UPS, F-31062 Toulouse, France. {\tt nechita@irsamc.ups-tlse.fr}}

\begin{abstract}
We discuss an extension of the almost Hadamard matrix formalism, to the case of complex matrices. Quite surprisingly, the situation here is very different from the one in the real case, and our conjectural conclusion is that there should be no such matrices, besides the usual Hadamard ones. We verify this conjecture in a number of situations, and notably for most of the known examples of real almost Hadamard matrices, and for some of their complex extensions. We discuss as well some potential applications of our conjecture, to the general study of complex Hadamard matrices.
\end{abstract}


\maketitle

{\em Dedicated to the memory of Uffe Haagerup.}

\tableofcontents

\section*{Introduction}

An Hadamard matrix is a square matrix $H\in M_N(\pm1)$, whose rows are pairwise orthogonal. Here is a basic example:
$$K_4=\begin{pmatrix}-1&1&1&1\\1&-1&1&1\\1&1&-1&1\\1&1&1&-1\end{pmatrix}$$

Assuming that the matrix has $N\geq3$ rows, the orthogonality conditions between the rows give $N\in4\mathbb N$. A similar analysis with four or more rows, or any other kind of abstract or concrete consideration doesn't give any further restriction on $N$, and we have:

\medskip

\noindent {\bf Hadamard Conjecture (HC).} {\em Hadamard matrices exist at any $N\in 4\mathbb N$.}

\medskip

This conjecture is about 100 years old. See \cite{had}, \cite{syl}.

Regarding the structure of the Hadamard matrices, the situation is complicated as well. As an example, the above matrix $K_4$ is circulant, obtained by cyclically permuting the entries of $v=(-1,1,1,1)$. So, as a first question, one may wonder whether one can fully classify the circulant Hadamard matrices. And the answer here is given by:

\medskip

\noindent {\bf Circulant Hadamard Conjecture (CHC).} {\em There are no circulant Hadamard matrices at $N>4$.}

\medskip

This conjecture is well-known too, and is about 50 years old. See \cite{rys}.

Generally speaking, the difficulty in dealing with such questions comes from the fact that the $\pm1$ entries can be replaced by any two symbols, with the orthogonality condition stating that, when comparing two rows, the number of matchings equals the number of mismatchings. We are therefore confronted to objects of the following type:
$$\begin{matrix}
\heartsuit&\heartsuit&\heartsuit&\heartsuit\\
\heartsuit&\clubsuit&\heartsuit&\clubsuit\\
\heartsuit&\heartsuit&\clubsuit&\clubsuit\\
\heartsuit&\clubsuit&\clubsuit&\heartsuit
\end{matrix}$$

Computers can of course help here, but only to some extent. There are as well connections to abstract algebra. The big challenge, however, remains that of inventing an efficient way of using classical analysis tools for the study of such objects.

While the difficulties abound, and no clear strategy is available, there have been many interesting advances on both the HC and the CHC, including:
\begin{enumerate}
\item Strong numeric evidence for these conjectures. See \cite{kta}, \cite{lsc}.

\item The recent theory of cocyclic Hadamard matrices. See \cite{hor}.

\item An asymptotic counting result for the partial Hadamard matrices \cite{lle}.
\end{enumerate}

Quite surprisingly, the landscape drastically changes when allowing the entries of $H$ to be roots of unity of arbitrary order, or even more generally, arbitrary complex numbers of modulus 1. These latter matrices are called ``complex Hadamard''. Such matrices exist at any $N$, the basic example being the Fourier matrix, $F_N=(w^{ij})_{ij}$ with $w=e^{2\pi i/N}$:
$$F_N
=\begin{pmatrix}
1&1&1&\ldots&1\\
1&w&w^2&\ldots&w^{N-1}\\
1&w^2&\omega^4&\ldots&w^{2(N-1)}\\
\ldots&\ldots&\ldots&\ldots&\ldots\\
1&w^{N-1}&w^{2(N-1)}&\ldots&w^{(N-1)^2}
\end{pmatrix}$$

Here the terminology comes from the fact that $F_N/\sqrt{N}$ is the matrix of the Fourier transform over the cyclic group $\mathbb Z_N$. More generally, associated to any finite abelian group $G$ is its Fourier matrix $F_G\in M_{|G|}(\mathbb C)$, which is complex Hadamard. As an example here, the above clubs and hearts design comes from the Klein group $\mathbb Z_2\times\mathbb Z_2$.

There are many other examples, often coming in tricky parametric families. In fact, the $N\times N$ complex Hadamard matrices form a real algebraic manifold $C_N\subset M_N(\mathbb C)$, which appears as an intersection of smooth manifolds, $C_N=M_N(\mathbb T)\cap\sqrt{N}\cdot U(N)$, and the known results and examples suggest that this intersection is highly singular.

The passage to the complex case brings as well a whole new range of potential motivations and applications. Generally speaking, the complex Hadamard matrices can be thought of as being ``generalized Fourier matrices'', and this is where the interest in them comes from. There are several potential applications of this philosophy, to various fields such as coding theory, operator algebras, quantum groups, quantum information, noncommutative geometry, linear algebra, abstract functional analysis. See \cite{tz1}.

Leaving aside these interpretations, what matters the most are of course the matrices themselves. In the case where the entries of the matrix are roots of unity of a given order, there has been some structure and classification  work, starting from the 60's, inspired from the real case \cite{but}. In the pure complex case, however, the systematic study started only quite recently, notably with two papers by Haagerup:
\begin{enumerate}
\item The classification up to $N=5$ was done in \cite{ha1}. The cases $N=2,3,4$ are elementary, but at $N=5$ the study requires a mix of ad-hoc techniques, of complex analysis and real algebraic geometry flavor, with $F_5$ being the only solution.

\item A counting result for the circulant complex Hadamard matrices of prime order was found in \cite{ha2}. The proof uses Bj\"orck's cyclic root picture \cite{bjo}, then basic ideas from algebraic geometry, and some number-theoretical ingredients.
\end{enumerate}

There have been several further developments of the subject, mainly via variations of these methods. See \cite{bbe}, \cite{bni}, \cite{kar}, \cite{szo}, \cite{tz2}. The structure of the complex Hadamard manifold remains, however, very unclear, and this even at $N=6$. Unclear as well is the relation between the local or global geometry of this manifold, and the above-mentioned collection of mathematical and physical ``generalized Fourier'' questions.

The present paper is a continuation of our previous work \cite{bcs}, \cite{bne}, \cite{bns}, \cite{bnz}. The starting point there was the fact that for $U\in O(N)$ we have, according to Cauchy-Schwarz:
$$||U||_1=\sum_{ij}|U_{ij}|\leq N\left(\sum_{ij}|U_{ij}|^2\right)^{1/2}=N\sqrt{N}$$

The equality case appears when the numbers $|U_{ij}|$ are all equal, so when $H=\sqrt{N}U$ is Hadamard. Motivated by this fact, we called almost Hadamard matrix (AHM) a matrix $H\in M_N(\mathbb R)$ having the property that $U=H/\sqrt{N}$ is a local maximizer of the 1-norm on $O(N)$. Such matrices exist at any $N\in\mathbb N$, the simplest example being:
$$K_N=\frac{1}{\sqrt{N}}
\begin{pmatrix}
2-N&2&\ldots&2\\
2&2-N&\ldots&2\\
\ldots&\ldots&\ldots&\ldots\\
2&2&\ldots&2-N
\end{pmatrix}$$

The AHM have a quite interesting structure, and as explained in \cite{bnz}, their construction requires some subtle combinatorial ingredients, such as the following object:
\begin{figure}[htbp]
\centering
\includegraphics[bb=10 1 250 275,width=0.3\textwidth]{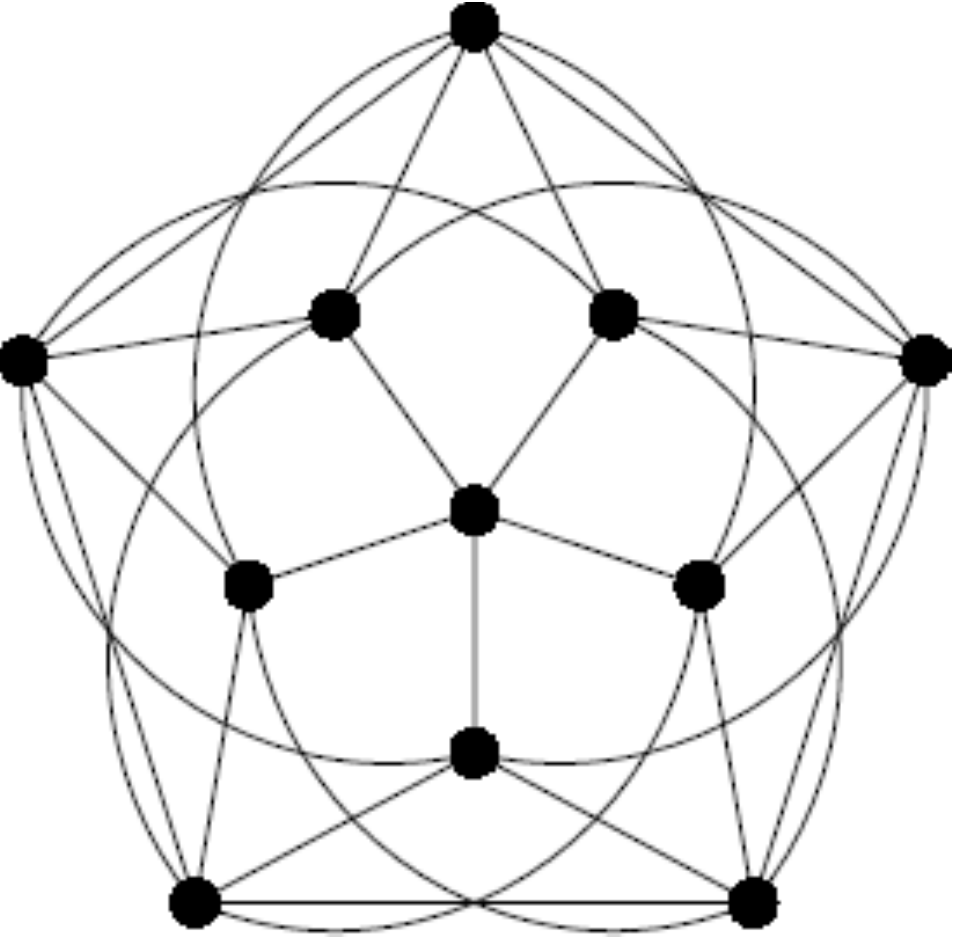}
\caption{The Paley biplane.}
\label{fig}
\end{figure}

Summarizing, the AHM theory appears as a natural relaxation of the Hadamard matrix theory, bringing a lot more freedom at the level of examples, and bringing into the picture some classical analysis as well. At the level of the potential applications, we first did some work in connection with the HC and CHC, with the conclusion (no surprise) that these questions are far too difficult. Then in \cite{bns} we found a first true application of our theory, stating that under suitable assumptions, the submatrices of Hadamard matrices are AHM. The consequences of this phenomenon are of course still to be explored.

In this paper we discuss an extension of the AHM formalism, to the case of complex matrices. We have not done this before, simply because the world of complex Hadamard matrices (CHM) looks quite rich already, and does not seem to ``need'' such an extension. However, as we will see, the complex AHM picture is in fact quite interesting.

As in the real case, the starting point is the Cauchy-Schwarz inequality $||U||_1\leq N\sqrt{N}$, but this time over the unitary group $U(N)$. The equality case happens when $H=\sqrt{N}U$ is CHM. Based on this observation, let us call complex AHM a matrix $H\in M_N(\mathbb C)$ having the property that $U=H/\sqrt{N}$ is a local maximizer of the 1-norm on $U(N)$.

Quite surprisingly, it is not clear at all on how to construct non-trivial examples, with $K_N$ and other basic AHM failing to be complex AHM. We have in fact the following statement, which emerges from our present work here:

\medskip

\noindent {\bf Almost Hadamard Conjecture (AHC). }{\em The only complex almost Hadamard matrices are the complex Hadamard matrices.}

\medskip

In other words, our conjecture is that a local maximizer of the 1-norm on $U(N)$ must be in fact a global maximizer. We will present here a number of verifications of this conjecture, with results regarding the following types of matrices:
\begin{enumerate}
\item The AHM coming from block designs.

\item The AHM which are circulant and symmetric.

\item The straightforward complex generalizations of such matrices.
\end{enumerate}

Regarding the potential applications, the situation is of course very different from the one in the real case. Assuming that the AHC holds indeed, we would have here a new approach to the CHM, which is by construction analytic and local. This would be quite powerful, with many potential applications. As an example here, numerical methods, such as the gradient descent one, could be used for finding new classes of CHM.

The main problem, however, remains that of proving the AHC, or at least finding a strategy for proving it. We will advance here on this question, with the conclusion that a potential proof might come via a clever mix of geometric and probabilistic techniques.

The paper is organized as follows: 1-2 are preliminary sections, in 3-4 we introduce the complex almost Hadamard matrices, in 5-6 we study in detail the unitary matrices coming from block designs, and in 7-8 we present a number of further verifications of our conjecture, and we discuss some potential consequences.

\medskip

\noindent {\bf Acknowledgments.} I.N.'s research has been supported by the ANR project {StoQ} {ANR-14-CE25-0003-01}.

\section{The Jensen inequality}

We are interested in this paper in the complex Hadamard matrices. The definition of these matrices is very simple, as follows:

\begin{definition}
A complex Hadamard matrix is a square matrix $H\in M_N(\mathbb C)$ whose entries are on the unit circle, $|H_{ij}|=1$, and whose rows are pairwise orthogonal. 
\end{definition}

The basic example is the Fourier matrix, $F_N=(w^{ij})_{ij}$ with $w=e^{2\pi i/N}$. This appears as matrix of the Fourier transform over the cyclic group $\mathbb Z_N$. Here are the first few such matrices, with the convention $i,j\in\{0,1,\ldots,N-1\}$, and with $w=e^{2\pi i/3}$:
$$F_2=\begin{pmatrix}1&1\\1&-1\end{pmatrix}\quad,\quad 
F_3=\begin{pmatrix}1&1&1\\ 1&w&w^2\\1&w^2&1\end{pmatrix}\quad,\quad 
F_4=\begin{pmatrix}1&1&1&1\\1&i&-1&-i\\1&-1&1&-1\\ 1&-i&-1&i\end{pmatrix}$$

In fact, associated to any finite abelian group $G$ is its Fourier matrix  $F_G\in M_{|G|}(\mathbb C)$. In terms of a decomposition $G=\mathbb Z_{N_1}\times\ldots\times\mathbb Z_{N_k}$ we have $F_G=F_{N_1}\otimes\ldots\otimes F_{N_K}$, and in particular we see that $F_G$ is a complex Hadamard matrix.

In general, a complex Hadamard matrix can be thought of as being a ``generalized Fourier matrix'', and this is where the interest in these matrices comes from. For a list of potential applications, for the most in connection with quantum physics, see \cite{tz1}.

Now back to Definition 1.1, observe that the orthogonality condition between the rows of $H$ tells us that the rescaled matrix $U=H/\sqrt{N}$ must belong to the unitary group $U(N)$. Following some previous observations, which go back to \cite{bcs}, in the real case, we have the following analytic characterization of such matrices:

\begin{proposition}
If $\psi:[0,\infty)\to\mathbb R$ is strictly concave/convex, the quantity
$$F(U)=\sum_{ij}\psi(|U_{ij}|^2)$$
over $U(N)$ is maximized/minimized precisely by the rescaled Hadamard matrices.
\end{proposition}

\begin{proof}
We recall that Jensen's inequality states that for $\psi$ convex we have:
$$\psi\left(\frac{x_1+\ldots+x_n}{n}\right)\leq\frac{\psi(x_1)+\ldots+\psi(x_n)}{n}$$

For $\psi$ concave the reverse inequality holds. Also, the equality case holds either when $\psi$ is linear, or when the numbers $x_1,\ldots,x_n$ are all equal.

In our case, with $n=N^2$ and with $\{x_1,\ldots,x_n\}=\{|U_{ij}|^2|i,j=1,\ldots,N\}$, we obtain that for any convex function $\psi$, the following holds:
$$\psi\left(\frac{1}{N}\right)\leq\frac{F(U)}{N^2}$$

Thus we have $F(U)\geq N^2\psi(1/N)$, and by assuming as in the statement that $\psi$ is strictly convex, the equality case holds precisely when the numbers $|U_{ij}|^2$ are all equal, so when $H=\sqrt{N}U$ is Hadamard. The proof for concave functions is similar.
\end{proof}

The above result suggests the following definition:

\begin{definition}
Given a concave/convex function $\psi:[0,\infty)\to\mathbb R$, we say that a matrix $H\in M_N(\mathbb C)$ is $\psi$-almost Hadamard if $U=H/\sqrt{N}$ belongs to $U(N)$, and $U$ locally maximizes/minimizes over $U(N)$ the following quantity:
$$F(U)=\sum_{ij}\psi(|U_{ij}|^2)$$
Also, we call $H$ absolute almost Hadamard if it is $\psi$-almost Hadamard, for any $\psi$.
\end{definition}

According to Proposition 1.2, any complex Hadamard matrix is an absolute almost Hadamard matrix. Our purpose here will be to study the converse of this fact.

Of particular interest for our considerations will be the power functions $\psi(x)=x^{p/2}$, which are concave at $p\in[1,2)$, and convex at $p\in(2,\infty)$. Observe that for such a function we have $F(U)=||U||_p^p$, where the $p$-norm is defined by the usual formula, namely:
$$||U||_p=\left(\sum_{ij}|U_{ij}|^p\right)^{1/p}$$
 
In particular, we can see that any absolute almost Hadamard matrix $H\in M_N(\mathbb C)$ must be such that $U=H/\sqrt{N}$ locally maximizes the $p$-norm on $U(N)$ at any $p\in[1,2)$, and locally minimizes the $p$-norm on $U(N)$ at any $p\in(2,\infty)$. 

In order to formulate now some classification results, we will need:

\begin{definition}
Two matrices $H,K\in M_N(\mathbb C)$ are called Hadamard equivalent if one can pass from one to the other via a composition of the following operations:
\begin{enumerate}
\item Permuting the rows, or permuting the columns.

\item Multiplying a row, or a column, by a number of modulus $1$.
\end{enumerate}
\end{definition}

At the level of classification results, it is known that, up to equivalence, the complex Hadamard matrices at $N=2,3,4,5$ are precisely the matrices $F_2,F_3,F_4^q,F_5$, where $F_4^q$ is a certain one-parameter deformation of the Fourier matrix $F_4$. See \cite{ha1}.

With this notion in hand, let us go back to the almost Hadamard matrices, and first study the case $N=2$. The situation here is very simple, as follows:

\begin{proposition}
At $N=2$ the various almost Hadamard notions coincide, and as example, we have only the Fourier matrix $F_2$, and its Hadamard conjugates.
\end{proposition}

\begin{proof}
We use the well-known fact that the unitary group $U(2)$ is given by:
$$U(2)=\left\{d\begin{pmatrix}a&b\\-\bar{b}&\bar{a}\end{pmatrix}\Big||d|=1,|a|^2+|b|^2=1\right\}$$

Let us pick $U\in U(2)$, written as above. For any $\psi:[0,\infty)\to\mathbb R$ we have then:
$$\sum_{ij}\psi(|U_{ij}|^2)=2\Big[\psi(|a|)+\psi(|b|)\Big]$$

It follows that when $\psi$ is strictly concave/convex, our matrix $U$ locally maximizes or minimizes this quantity precisely when $|a|=|b|$. We conclude that any type of ``almost Hadamard'' condition on $H=\sqrt{2}U$ requires $U$ to be as follows:
$$U=\frac{d}{\sqrt{2}}\begin{pmatrix}\alpha&\beta\\-\bar{\beta}&\bar{\alpha}\end{pmatrix}\quad,\quad|d|=|\alpha|=|\beta|=1$$

Now observe that this matrix is rescaled complex Hadamard. Thus by \cite{ha1} the matrix $H=\sqrt{2}U$ must be Hadamard equivalent to the Fourier matrix $F_2$, and we are done.
\end{proof}

The following key fact, which in the real case goes back to \cite{bcs}, is crucial in the study of almost Hadamard matrices:

\begin{theorem}
If $U\in U(N)$ locally maximizes over $U(N)$ the quantity
$$||U||_1=\sum_{ij}|U_{ij}|$$
then all its entries are nonzero, $U_{ij}\neq0$ for any $i,j$.
\end{theorem}

\begin{proof}
We use the same method as in the real case \cite{bcs}, namely a ``rotation trick''. Let us denote by $U_1,\ldots,U_N$ the rows of $U$, and let us perform a rotation of $U_1,U_2$:
$$\begin{bmatrix}U^t_1\\ U^t_2\end{bmatrix}
=\begin{bmatrix}\cos t\cdot U_1-\sin t\cdot U_2\\ \sin t\cdot U_1+\cos t\cdot U_2\end{bmatrix}$$

In order to compute the 1-norm, let us permute the columns of $U$, in such a way that the first two rows look as follows, with $X,Y,A,B$ having nonzero entries:
$$\begin{bmatrix}U_1\\ U_2\end{bmatrix}
=\begin{bmatrix}0&0&Y&A\\0&X&0&B\end{bmatrix}$$

The rotated matrix will look then as follows:
$$\begin{bmatrix}U_1^t\\ U_2^t\end{bmatrix}
=\begin{bmatrix}
0&-\sin t\cdot X&\cos t\cdot Y&\cos t\cdot A-\sin t\cdot B\\
0&\cos t\cdot X&\sin t\cdot y&\sin t\cdot A+\cos t\cdot B\end{bmatrix}$$

Our claim is that $X,Y$ must be empty. Indeed, if $A$ and $B$ are not empty, let us fix a column index $k$ for both $A,B$, and set $\alpha=A_k$, $\beta=B_k$. We have then:
\begin{eqnarray*}
|(U_1^t)_k|+|(U_2^t)_k|
&=&|\cos t\cdot\alpha-\sin t\cdot\beta|+|\sin t\cdot\alpha+\cos t\cdot\beta|\\
&=&\sqrt{\cos^2t\cdot|\alpha|^2+\sin^2t\cdot|\beta|^2-\sin t\cos t(\alpha\bar{\beta}+\beta\bar{\alpha})}\\
&+&\sqrt{\sin^2t\cdot|\alpha|^2+\cos^2t\cdot|\beta|^2+\sin t\cos t(\alpha\bar{\beta}+\beta\bar{\alpha})}
\end{eqnarray*}

Since $\alpha,\beta\neq 0$, the above function is derivable at $t=0$, and we obtain:
\begin{eqnarray*}
\frac{\partial\left(|(U_1^t)_k|+|(U_2^t)_k|\right)}{\partial t}
&=&\frac{\sin 2t(|\beta|^2-|\alpha|^2)-\cos 2t(\alpha\bar{\beta}+\beta\bar{\alpha})}{2\sqrt{\cos^2t\cdot|\alpha|^2+\sin^2t\cdot|\beta|^2-\sin t\cos t(\alpha\bar{\beta}+\beta\bar{\alpha})}}\\
&+&\frac{\sin 2t(|\alpha|^2-|\beta|^2)+\cos 2t(\alpha\bar{\beta}+\beta\bar{\alpha})}{2\sqrt{\sin^2t\cdot|\alpha|^2+\cos^2t\cdot|\beta|^2+\sin t\cos t(\alpha\bar{\beta}+\beta\bar{\alpha})}}
\end{eqnarray*}

Thus at $t=0$, we obtain the following formula:
$$\frac{\partial\left(|(U_1^t)_k|+|(U_2^t)_k|\right)}{\partial t}(0)=\frac{\alpha\bar{\beta}+\beta\bar{\alpha}}{2}\left(\frac{1}{|\beta|}-\frac{1}{|\alpha|}\right)$$

Now since $U$ locally maximizes the 1-norm, both directional derivatives of $||U^t||_1$ must be negative in the limit $t\to 0$. On the other hand, if we denote by $C$ the contribution coming from the right (which might be zero in the case where $A$ and $B$ are empty), i.e. the sum over $k$ of the above quantities, we have:
\begin{eqnarray*}
\frac{\partial||U^t||_1}{\partial t}_{\big|t=0^+}
&=&\frac{\partial}{\partial t}_{\big|t=0^+}(|\cos t|+|\sin t|)(||X||_1+||Y||_1)+C\\
&=&(-\sin t + \cos t)_{\big|t=0}(||X||_1+||Y||_1)+C\\
&=&||X||_1+||Y||_1+C
\end{eqnarray*}

As for the derivative at left, this is given by the following formula:
\begin{eqnarray*}
\frac{\partial||U^t||_1}{\partial t}_{\big|t=0^-}
&=&\frac{\partial}{\partial t}_{\big|t=0^-}(|\cos t|+|\sin t|)(||X||_1+||Y||_1)+C\\
&=&(-\sin t - \cos t)_{\big|t=0}(||X||_1+||Y||_1)+C\\
&=&-||X||_1-||Y||_1+C
\end{eqnarray*}

We therefore obtain the following inequalities, where $C$ is as above:
\begin{eqnarray*}
||X||_1+||Y||_1+C &\leq& 0\\
-||X||_1-||Y||_1+C&\leq& 0
\end{eqnarray*}

Consider now the matrix obtained from $U$ by interchanging $U_1,U_2$. Since this matrix must be as well a local maximizer of the 1-norm, and since the above formula shows that $C$ changes its sign when interchanging $U_1,U_2$, we obtain:
\begin{eqnarray*}
||X||_1+||Y||_1-C &\leq& 0\\
-||X||_1-||Y||_1-C&\leq& 0
\end{eqnarray*}

The four inequalities that we have give altogether $||X||_1+||Y||_1=C=0$, and from $||X||_1+||Y||_1=0$ we obtain that both $X,Y$ must be empty, as claimed.

As a conclusion, up to a permutation of the columns, the first two rows must be of the following form, with $A,B$ having only nonzero entries:
$$\begin{bmatrix}U_1\\ U_2\end{bmatrix}
=\begin{bmatrix}0&A\\0&B\end{bmatrix}$$

By permuting the rows of $U$, the same must hold for any two rows $U_i,U_j$. Now since $U$ cannot have a zero column, we conclude that $U$ cannot have zero entries, as claimed.
\end{proof}

As explained in \cite{bne}, a $p$-norm analogue of the above result holds in the real case, with $p<2$. The extension of this result to the complex case, as well as the generalization to exponents $p>2$, or to arbitrary convex/concave functions, remains an open problem.

Yet another interesting question regards the local minimizers of the 1-norm. It is elementary to see that the global minimizers of the 1-norm are the generalized permutation matrices (i.e. the matrices $U\in U(N)$ having a maximal number of 0 entries), but at the level of local minimizers of the 1-norm, we have no results, so far.

\section{Critical points}

We denote by $U(N)^*$ the set of matrices $U\in U(N)$ having nonzero entries. In view of Theorem 1.6 above, in order to investigate the one-norm almost Hadamard matrices, or the absolute ones, we can restrict the attention to the matrices $U\in U(N)^*$.

Our first task will be that of investigating the critical points over $U(N)^*$ of the various functions of type $F(U)=\sum_{ij}\psi(|U_{ij}|^2)$. We focus here on the first order, not taking into account the convexity/concavity properties of $\psi$, and it is technically convenient to use the function $\varphi(x)=\psi(x^2)$, with no extra assumptions on it. 

Following some previous work from \cite{bcs}, \cite{bne}, we first have:

\begin{proposition}
Let $\varphi:[0,\infty)\to\mathbb R$ be a differentiable function. A matrix $U\in U(N)^*$ is a critical point of the quantity
$$F(U)=\sum_{ij}\varphi(|U_{ij}|)$$
precisely when $WU^*$ is self-adjoint, where $W_{ij}=sgn(U_{ij})\varphi'(|U_{ij}|)$.
\end{proposition}

\begin{proof}
We regard $U(N)$ as a real algebraic manifold, with coordinates $U_{ij},\bar{U}_{ij}$. This manifold consists by definition of the zeroes of the following polynomials: $$A_{ij}=\sum_kU_{ik}\bar{U}_{jk}-\delta_{ij}$$

Since $U(N)$ is smooth, and so is a differential manifold in the usual sense, it follows from the general theory of Lagrange multipliers that a given matrix $U\in U(N)$ is a critical point of $F$ precisely when the condition $dF\in span(dA_{ij})$ is satisfied. 

Regarding the space $span(dA_{ij})$, this consists of the following quantities:
\begin{eqnarray*}
\sum_{ij}M_{ij}dA_{ij}
&=&\sum_{ijk}M_{ij}(U_{ik}d\bar{U}_{jk}+\bar{U}_{jk}dU_{ik})\\
&=&\sum_{jk}(M^tU)_{jk}d\bar{U}_{jk}+\sum_{ik}(M\bar{U})_{ik}dU_{ik}\\
&=&\sum_{ij}(M^tU)_{ij}d\bar{U}_{ij}+\sum_{ij}(M\bar{U})_{ij}dU_{ij}
\end{eqnarray*}

In order to compute $dF$, observe first that, with $S_{ij}=sgn(U_{ij})$, we have:
$$d|U_{ij}|=d\sqrt{U_{ij}\bar{U}_{ij}}=\frac{U_{ij}d\bar{U}_{ij}+\bar{U}_{ij}dU_{ij}}{2|U_{ij}|}=\frac{1}{2}(S_{ij}d\bar{U}_{ij}+\bar{S}_{ij}dU_{ij})$$

We therefore obtain, with $W_{ij}=sgn(U_{ij})\varphi'(|U_{ij}|)$ as in the statement:
$$dF=\sum_{ij}d\left(\varphi(|U_{ij}|)\right)=\sum_{ij}\varphi'(|U_{ij}|)d|U_{ij}|=\frac{1}{2}\sum_{ij}W_{ij}d\bar{U}_{ij}+\bar{W}_{ij}dU_{ij}$$

We conclude that $U\in U(N)$ is a critical point of $F$ if and only if there exists a matrix $M\in M_N(\mathbb C)$ such that the following two conditions are satisfied:
$$W=2M^tU\quad,\quad\bar{W}=2M\bar{U}$$

Now observe that these two equations can be written as follows:
$$M^t=\frac{1}{2}WU^*\quad,\quad M^t=\frac{1}{2}UW^*$$

Summing up, the critical point condition on $U\in U(N)$ simply reads $WU^*=UW^*$, which means that the matrix $WU^*$ must be self-adjoint, as claimed.
\end{proof}

In order to process the above result, use the following notion, from \cite{bne}:

\begin{definition}
The color decomposition of a matrix $U\in M_N(\mathbb C)$ is $U=\sum_{r>0}rU_r$, where $U_r\in M_N(\mathbb T\cup\{0\})$ are the matrices given by
$$(U_r)_{ij}=\begin{cases}
sgn(U_{ij})&{\rm if}\ |U_{ij}|=r\\
0&{\rm otherwise}
\end{cases}$$
which describe where the various modulus $r$ entries stand.
\end{definition}

The terminology comes from the fact that for certain applications, as those that we will need here, the values of the various numbers $r>0$ which appear inside $U$ are most of the time irrelevant, so we can think of these entries rather as being ``colors''.

We can now introduce the following notions:

\begin{definition}
We call a unitary matrix $U\in U(N)$:
\begin{enumerate}
\item Semi-balanced, if the matrices $U_rU^*$ and $U^*U_r$, with $r>0$, are all self-adjoint.

\item Balanced, if the matrices $U_rU_s^*$ and $U_r^*U_s$, with $r,s>0$, are all self-adjoint.
\end{enumerate}
\end{definition}

These conditions are quite natural, because for a unitary matrix $U\in U(N)$, the relations $UU^*=U^*U=1$ translate as follows, in terms of the color decomposition:
$$\sum_{r>0}rU_rU^*=\sum_{r>0}rU^*U_r=1$$
$$\sum_{r,s>0}rsU_rU_s^*=\sum_{r,s>0}rsU_r^*U_s=1$$

Thus, our balancing conditions express the fact that the various components of the above sums all self-adjoint. Now back to our critical point questions, we have:

\begin{theorem}
The joint critical points $U\in U(N)^*$ of the functions
$$F(U)=\sum_{ij}\varphi(|U_{ij}|)$$
with $\varphi:(0,\infty)\to\mathbb R$, are precisely the semi-balanced matrices.
\end{theorem}

\begin{proof}
We use Proposition 2.1 above. The matrix constructed there is given by:
\begin{eqnarray*}
(WU^*)_{ij}
&=&\sum_k{\rm sgn}(U_{ik})\varphi'(|U_{ik}|)\bar{U}_{jk}\\
&=&\sum_{r>0}\varphi'(r)\sum_{k,|U_{ik}|=r}{\rm sgn}(U_{ik})\bar{U}_{jk}\\
&=&\sum_{r>0}\varphi'(r)\sum_k(U_r)_{ik}\bar{U}_{jk}\\
&=&\sum_{r>0}\varphi'(r)(U_rU^*)_{ij}
\end{eqnarray*}

Thus we have $WU^*=\sum_{r>0}\varphi'(r)U_rU^*$, and when $\varphi:(0,\infty)\to\mathbb R$ varies, the individual components of this sum must be all self-adjoint, as claimed.
\end{proof}

As a conclusion, algebrically speaking, we are led to the study of the semi-balanced matrices. The point, however, is that most of the known examples of semi-balanced matrices are actually balanced. So, while the analytic meaning of the balancing condition remains quite unclear, we would like now to present a few results on this class of matrices, which seems to be a quite interesting one, from a combinatorial point of view. 

As a first result, we have the following collection of simple facts:

\begin{proposition}
The class of unitary balanced matrices is as follows:
\begin{enumerate}
\item It contains the matrices $U=H/\sqrt{N}$, with $H\in M_N(\mathbb C)$ Hadamard.

\item It is stable under transposition, complex conjugation, and taking adjoints.

\item It is stable under taking tensor products.

\item It is stable under the Hadamard equivalence relation.

\item It contains the matrix $U_N=\frac{1}{N}(2\mathbb I_N-N1_N)$, where $\mathbb I_N$ is the all-$1$ matrix.
\end{enumerate}
\end{proposition}

\begin{proof}
All these results are elementary, the proof being as follows:

(1) Here $U\in U(N)$ follows from the Hadamard condition, and since there is only one color component, namely $U_{1/\sqrt{N}}=H$, the balancing condition is satisfied as well.

(2) Assuming that $U=\sum_{r>0}rU_r$ is a color decomposition of a given matrix $U\in U(N)$, the following are color decompositions too:
$$U^t=\sum_{r>0}rU_r^t\quad,\quad\bar{U}=\sum_{r>0}r\bar{U}_r\quad,\quad U^*=\sum_{r>0}rU_r^*$$

But this observation gives all the assertions. 

(3) Assuming that $U=\sum_{r>0}rU_r$ and $V=\sum_{s>0}sV_s$ are the color decompositions of two given unitary matrices $U,V$, we have:
$$U\otimes V=\sum_{r,s>0}rs\cdot U_r\otimes V_s=\sum_{p>0}p\sum_{p=rs}U_r\otimes V_s$$

Thus the color components of $W=U\otimes V$ are the matrices $W_p=\sum_{p=rs}U_r\otimes V_s$, and it follows that if $U,V$ are both balanced, then so is $W=U\otimes V$.

(4) We recall that the Hadamard equivalence consists in permuting rows and columns, and switching signs on rows and columns. Since all these operations correspond to certain conjugations at the level of the matrices $U_rU_s^*,U_r^*U_s$, we obtain the result.

(5) The matrix in the statement, which goes back to \cite{bnz}, is as follows:
$$U_N=\frac{1}{N}
\begin{pmatrix}
2-N&2&\ldots&2\\
2&2-N&\ldots&2\\
\ldots&\ldots&\ldots&\ldots\\
2&2&\ldots&2-N
\end{pmatrix}$$

Observe that this matrix is indeed unitary, its rows being of norm one, and pairwise orthogonal. The color components of this matrix being $U_{2/N-1}=1_N$ and $U_{2/N}=\mathbb I_N-1_N$, it follows that this matrix is balanced as well, as claimed.
\end{proof}

Let us look now more in detail at $U_N$, and at the matrices having similar properties. We recall from \cite{bnz} that an $(a,b,c)$ pattern is a matrix $M\in M_N(0,1)$, with $N=a+2b+c$, such that any two rows look as follows, up to a permutation of the columns:
$$\begin{matrix}
0\ldots 0&0\ldots 0&1\ldots 1&1\ldots 1\\
\underbrace{0\ldots 0}_a&\underbrace{1\ldots 1}_b&\underbrace{0\ldots 0}_b&\underbrace{1\ldots 1}_c
\end{matrix}$$

As explained in \cite{bnz}, there are many interesting examples of $(a,b,c)$ patterns, coming from the balanced incomplete block designs (BIBD), and all these examples can produce two-entry unitary matrices, by replacing the $0,1$  entries with suitable numbers $x,y$. 

Now back to the matrix $U_N$ from Proposition 2.5 (5), observe that this matrix comes from a $(0,1,N-2)$ pattern. And also, independently of this, this matrix has the remarkable property of being at the same time circulant and self-adjoint. 

We have in fact the following result, generalizing Proposition 2.5 (4):

\begin{proposition}
The following matrices are balanced:
\begin{enumerate}
\item The orthogonal matrices coming from $(a,b,c)$ patterns.

\item The unitary matrices which are circulant and self-adjoint.
\end{enumerate}
\end{proposition}

\begin{proof}
These observations basically go back to \cite{bnz}, and then to \cite{bne}, in the real case. In the general case, the proofs are as follows:

(1) If we denote by $P,Q\in M_N(0,1)$ the matrices describing the positions of the $0,1$ entries inside the pattern, then we have the following formulae:
\begin{eqnarray*}
PP^t=P^tP&=&a\mathbb I_N+b1_N\\
QQ^t=Q^tQ&=&c\mathbb I_N+b1_N\\
PQ^t=P^tQ=QP^t=Q^tP&=&b\mathbb I_N-b1_N
\end{eqnarray*}

Since all these matrices are symmetric, $U$ is balanced, as claimed.

(2) Assume that $U\in U(N)$ is circulant, $U_{ij}=\gamma_{j-i}$, and in addition self-adjoint, which means $\bar{\gamma}_i=\gamma_{-i}$. Consider the following sets, which must satisfy $D_r=-D_r$:
$$D_r=\{k:|\gamma_r|=k\}$$

In terms of these sets, we have the following formula:
\begin{eqnarray*}
(U_rU_s^*)_{ij}
&=&\sum_k(U_r)_{ik}(\bar{U}_s)_{jk}\\
&=&\sum_k\delta_{|\gamma_{k-i}|,r}\,sgn(\gamma_{k-i})\cdot\delta_{|\gamma_{k-j}|,s}\,sgn(\bar{\gamma}_{k-j})\\
&=&\sum_{k\in(D_r+i)\cap(D_s+j)}sgn(\gamma_{k-i})sgn(\bar{\gamma}_{k-j})
\end{eqnarray*}

With $k=i+j-m$ we obtain, by using $D_r=-D_r$, and then $\bar{\gamma}_i=\gamma_{-i}$:
\begin{eqnarray*}
(U_rU_s^*)_{ij}
&=&\sum_{m\in(-D_r+j)\cap(-D_s+i)}sgn(\gamma_{j-m})sgn(\bar{\gamma}_{i-m})\\
&=&\sum_{m\in(D_r+i)\cap(D_r+j)}sgn(\gamma_{j-m})sgn(\bar{\gamma}_{i-m})\\
&=&\sum_{m\in(D_r+i)\cap(D_r+j)}sgn(\bar{\gamma}_{m-j})sgn(\gamma_{m-i})
\end{eqnarray*}

Now by interchanging $i\leftrightarrow j$, and with $m\to k$, this formula becomes:
$$(U_rU_s^*)_{ji}=\sum_{k\in(D_r+i)\cap(D_r+j)}sgn(\bar{\gamma}_{k-i})sgn(\gamma_{k-j})$$

We recognize here the complex conjugate of $(U_rU_s^*)_{ij}$, as previously computed above, and we therefore deduce that $U_rU_s^*$ is self-adjoint. The proof for $U_r^*U_s$ is similar. 
\end{proof}

There are several interesting questions regarding the balanced unitary matrices. A first question is that of understanding the precise analytic meaning of these matrices, say as critical points of some cleverly chosen functions on $U(N)$. A second question is that of understanding the precise combinatorial meaning of these matrices, in the general context of design theory \cite{cdi}, \cite{sti}. Finally, a third question regards the general structure and classification of such matrices, for instance at small values of $N$.

\section{Hessian computations}

Let us go back now to the Jensen inequality from Proposition 1.2 above, and to the quantities $F(U)=\sum_{ij}\psi(|U_{ij}|^2)$ appearing there. In order to study the local extrema of these quantitites, consider the following function, depending on $t>0$ small:
$$f(t)=F(Ue^{tA})=\sum_{ij}\psi(|(Ue^{tA})_{ij}|^2)$$

Here $U\in U(N)$ is an arbitrary unitary, and $A\in M_N(\mathbb C)$ is assumed to be anti-hermitian, $A^*=-A$, with this latter assumption needed for having $e^A\in U(N)$.

Let us first compute the derivative of $f$. We have:

\begin{proposition}
We have the following formula,
$$f'(t)=2\sum_{ij}\psi'(|(Ue^{tA})_{ij}|^2)Re\left[(UAe^{tA})_{ij}\overline{(Ue^{tA})_{ij}}\right]$$
valid for any $U\in U(N)$, and any $A\in M_N(\mathbb C)$ anti-hermitian.
\end{proposition}

\begin{proof}
The matrices $U,e^{tA}$ being both unitary, we have:
\begin{eqnarray*}
|(Ue^{tA})_{ij}|^2
&=&(Ue^{tA})_{ij}\overline{(Ue^{tA})_{ij}}\\
&=&(Ue^{tA})_{ij}((Ue^{tA})^*)_{ji}\\
&=&(Ue^{tA})_{ij}(e^{tA^*}U^*)_{ji}\\
&=&(Ue^{tA})_{ij}(e^{-tA}U^*)_{ji}
\end{eqnarray*}

We can now differentiate our function $f$, and by using once again the unitarity of the matrices $U,e^{tA}$, along with the formula $A^*=-A$, we obtain:
\begin{eqnarray*}
f'(t)
&=&\sum_{ij}\psi'(|(Ue^{tA})_{ij}|^2)\left[(UAe^{tA})_{ij}(e^{-tA}U^*)_{ji}-(Ue^{tA})_{ij}(e^{-tA}AU^*)_{ji}\right]\\
&=&\sum_{ij}\psi'(|(Ue^{tA})_{ij}|^2)\left[(UAe^{tA})_{ij}\overline{((e^{-tA}U^*)^*)_{ij}}-(Ue^{tA})_{ij}\overline{((e^{-tA}AU^*)^*)_{ij}}\right]\\
&=&\sum_{ij}\psi'(|(Ue^{tA})_{ij}|^2)\left[(UAe^{tA})_{ij}\overline{(Ue^{tA})_{ij}}+(Ue^{tA})_{ij}\overline{(UAe^{tA})_{ij}}\right]
\end{eqnarray*}

But this gives the formula in the statement, and we are done.
\end{proof}

Before computing the second derivative, let us evaluate $f'(0)$. In terms of the color decomposition $U=\sum_{r>0}rU_r$ of our matrix, the result is as follows:

\begin{proposition}
We have the following formula,
$$f'(0)=2\sum_{r>0}r\psi'(r^2)Re\left[Tr(U_r^*UA)\right]$$
where $U_r\in M_N(\mathbb T\cup\{0\})$ are the color components of $U$.
\end{proposition}

\begin{proof}
We use the formula in Proposition 3.1 above. At $t=0$, we obtain:
$$f'(0)=2\sum_{ij}\psi'(|U_{ij}|^2)Re\left[(UA)_{ij}\overline{U}_{ij}\right]$$

Consider now the color decomposition of $U$. We have the following formulae:
\begin{eqnarray*}
U_{ij}=\sum_{r>0}r(U_r)_{ij}
&\implies&|U_{ij}|^2=\sum_{r>0}r^2|(U_r)_{ij}|\\
&\implies&\psi'(|U_{ij}|^2)=\sum_{r>0}\psi'(r^2)|(U_r)_{ij}|
\end{eqnarray*}

Now by getting back to the above formula of $f'(0)$, we obtain:
$$f'(0)=2\sum_{r>0}\psi'(r^2)\sum_{ij}Re\left[(UA)_{ij}\overline{U}_{ij}|(U_r)_{ij}|\right]$$

Our claim now is that we have $\overline{U}_{ij}|(U_r)_{ij}|=r\overline{(U_r)}_{ij}$. Indeed, in the case $|U_{ij}|\neq r$ this formula reads $\overline{U}_{ij}\cdot 0=r\cdot 0$, which is true, and in the case $|U_{ij}|=r$ this formula reads $r\bar{S}_{ij}\cdot 1=r\cdot\bar{S}_{ij}$, which is once again true. We therefore conclude that we have:
$$f'(0)=2\sum_{r>0}r\psi'(r^2)\sum_{ij}Re\left[(UA)_{ij}\overline{(U_r)}_{ij}\right]$$

But this gives the formula in the statement, and we are done.
\end{proof}

As an illustration, for the function $\psi(x)=\sqrt{x}$, we obtain:
$$f'(0)=\sum_{r>0}Re\left[Tr(U_r^*UA)\right]=Re\left[Tr(S^*UA)\right]=\frac{1}{2}Tr\left[(S^*U-U^*S)A\right]$$

We conclude that the critical point condition, namely $f'(0)=0$ for any $A$ anti-hermitian, is equivalent to $S^*U=U^*S$, and so to the fact that $S^*U$ is self-adjoint.

In general, we recover of course the algebraic results from section 2 above.

Let us compute now the second derivative. The result here is as follows:

\begin{proposition}
We have the following formula,
\begin{eqnarray*}
f''(0)
&=&4\sum_{ij}\psi''(|U_{ij}|^2)Re\left[(UA)_{ij}\overline{U}_{ij}\right]^2\\
&&+2\sum_{ij}\psi'(|U_{ij}|^2)Re\left[(UA^2)_{ij}\overline{U}_{ij}\right]\\
&&+2\sum_{ij}\psi'(|U_{ij}|^2)|(UA)_{ij}|^2
\end{eqnarray*}
valid for any $U\in U(N)$, and any $A\in M_N(\mathbb C)$ anti-hermitian.
\end{proposition}

\begin{proof}
We use the formula in Proposition 3.1 above, namely:
$$f'(t)=2\sum_{ij}\psi'(|(Ue^{tA})_{ij}|^2)Re\left[(UAe^{tA})_{ij}\overline{(Ue^{tA})_{ij}}\right]$$

Since the real part on the right, or rather its double, appears as the derivative of the quantity $|(Ue^{tA})_{ij}|^2$, when differentiating a second time, we obtain:
\begin{eqnarray*}
f''(t)
&=&4\sum_{ij}\psi''(|(Ue^{tA})_{ij}|^2)Re\left[(UAe^{tA})_{ij}\overline{(Ue^{tA})_{ij}}\right]^2\\
&&+2\sum_{ij}\psi'(|(Ue^{tA})_{ij}|^2)Re\left[(UAe^{tA})_{ij}\overline{(Ue^{tA})_{ij}}\right]'
\end{eqnarray*}

In order to compute now the missing derivative, observe that we have:
\begin{eqnarray*}
\left[(UAe^{tA})_{ij}\overline{(Ue^{tA})_{ij}}\right]'
&=&(UA^2e^{tA})_{ij}\overline{(Ue^{tA})_{ij}}+(UAe^{tA})_{ij}\overline{(UAe^{tA})_{ij}}\\
&=&(UA^2e^{tA})_{ij}\overline{(Ue^{tA})_{ij}}+|(UAe^{tA})_{ij}|^2
\end{eqnarray*}

Summing up, we have obtained the following formula:
\begin{eqnarray*}
f''(t)
&=&4\sum_{ij}\psi''(|(Ue^{tA})_{ij}|^2)Re\left[(UAe^{tA})_{ij}\overline{(Ue^{tA})_{ij}}\right]^2\\
&&+2\sum_{ij}\psi'(|(Ue^{tA})_{ij}|^2)Re\left[(UA^2e^{tA})_{ij}\overline{(Ue^{tA})_{ij}}\right]\\
&&+2\sum_{ij}\psi'(|(Ue^{tA})_{ij}|^2)|(UAe^{tA})_{ij}|^2
\end{eqnarray*}

But at $t=0$ this gives the formula in the statement, and we are done.
\end{proof}

For the function $\psi(x)=\sqrt{x}$, corresponding to the functional $F(U)=||U||_1$, there are some simplifications, that we will work out now in detail. First, we have:

\begin{proposition}
Let $U \in U(N)^*$. For the function $F(U)=||U||_1$ we have the formula
$$f''(0)=Re\left[Tr(S^*UA^2)\right]+\sum_{ij}\frac{Im\left[(UA)_{ij}\overline{S}_{ij}\right]^2}{|U_{ij}|}$$
valid for any anti-hermitian matrix $A$, where $U_{ij}=S_{ij}|U_{ij}|$.
\end{proposition}

\begin{proof}
We use the formula in Proposition 3.3 above, with $\psi(x)=\sqrt{x}$. The derivatives are here $\psi'(x)=\frac{1}{2\sqrt{x}}$ and $\psi''(x)=-\frac{1}{4x\sqrt{x}}$, and we obtain:
\begin{eqnarray*}
f''(0)
&=&-\sum_{ij}\frac{Re\left[(UA)_{ij}\overline{U}_{ij}\right]^2}{|U_{ij}|^3}
+\sum_{ij}\frac{Re\left[(UA^2)_{ij}\overline{U}_{ij}\right]}{|U_{ij}|}
+\sum_{ij}\frac{|(UA)_{ij}|^2}{|U_{ij}|}\\
&=&-\sum_{ij}\frac{Re\left[(UA)_{ij}\overline{S}_{ij}\right]^2}{|U_{ij}|}
+\sum_{ij}Re\left[(UA^2)_{ij}\overline{S}_{ij}\right]
+\sum_{ij}\frac{|(UA)_{ij}|^2}{|U_{ij}|}\\
&=&Re\left[Tr(S^*UA^2)\right]+\sum_{ij}\frac{|(UA)_{ij}|^2-Re\left[(UA)_{ij}\overline{S}_{ij}\right]^2}{|U_{ij}|}
\end{eqnarray*}

But this gives the formula in the statement, and we are done.
\end{proof}

We are therefore led to the following result, regarding the 1-norm:

\begin{theorem}
A matrix $U\in U(N)^*$ locally maximizes the one-norm on $U(N)$ precisely when $S^*U$ is self-adjoint, where $S_{ij}=sgn(U_{ij})$, and when
$$Tr(S^*UA^2)+\sum_{ij}\frac{Im\left[(UA)_{ij}\overline{S}_{ij}\right]^2}{|U_{ij}|}\leq0$$
holds, for any anti-hermitian matrix $A\in M_N(\mathbb C)$.
\end{theorem}

\begin{proof}
According to Proposition 2.1 and Proposition 3.4, the local maximizer condition requires $X=S^*U$ to be self-adjoint, and the following inequality to be satisfied:
$$Re\left[Tr(S^*UA^2)\right]+\sum_{ij}\frac{Im\left[(UA)_{ij}\overline{S}_{ij}\right]^2}{|U_{ij}|}\leq0$$

Now observe that since both $X$ and $A^2$ are self-adjoint, we have:
$$Re\left[Tr(XA^2)\right]=\frac{1}{2}\left[Tr(XA^2)+Tr(A^2X)\right]=Tr(XA^2)$$

Thus we can remove the real part, and we obtain the inequality in the statement.
\end{proof}

As a general comment, all the above computations can be of course interpreted by using more advanced geometric language. The unitary group $U(N)$ is a Lie group, and its tangent space at $U\in U(N)$ is isomorphic to the corresponding Lie algebra, which consists of the anti-hermitian matrices $A\in M_N(\mathbb C)$. With this picture in hand, our formulae for $f'(0)$ translate into the fact that the gradient of the 1-norm is given by:
$$\nabla||U||_1=\frac{1}{2}(S-US^*U)$$

Regarding now the second derivative, $f''(0)$, our computations here provide us with a formula for the Hessian of the 1-norm. Indeed, with the change of variables $A=iB$ on the tangent space, the Hessian is given by $<B,H(B)>=-\Phi(U,B)$, where $\Phi(U,iA)$ is the quantity appearing in Theorem 3.5. In order to explicitely compute now $H$, it is enough to apply to our formula the usual polarization identity, namely:
$$<A,H(B)>=\frac{1}{2}\left[<A+B,H(A+B)>-<A,H(A)>-<B, H(B)>\right]$$

We obtain that $H$ is given by the following formula, with $A,B \in M_N^{sa}(\mathbb C)$:
$$<A,H(B)>=-\frac{1}{2}Tr[S^*U(AB+BA)]+\sum_{ij}\frac{Re\left[(UA)_{ij}\overline{S}_{ij}\right] Re\left[(UB)_{ij}\overline{S}_{ij}\right]}{|U_{ij}|}$$

We will be back to more advanced geometric considerations in section 8 below. 

\section{Almost Hadamard matrices}

Starting from this section, we restrict attention to the one-norm. We will be interested in what follows in the following type of matrices:

\begin{definition}
A matrix $H\in M_N(\mathbb C)$ is called complex almost Hadamard if $U=H/\sqrt{N}$ is unitary, and locally maximizes the $1$-norm on $U(N)$.
\end{definition}

We already know that any complex Hadamard matrix $H\in M_N(\mathbb C)$ is almost Hadamard, because its rescaling $U=H/\sqrt{N}$ globally maximizes the $1$-norm on $U(N)$. This follows indeed from Proposition 1.2 above, or simply from Cauchy-Schwarz, as follows:
$$||U||_1=\sum_{ij}|U_{ij}|\leq N\left(\sum_{ij}|U_{ij}|^2\right)^{1/2}=N\sqrt{N}$$

Let us mention right away that our goal in what follows will be that of providing evidence for the following conjecture:

\begin{conjecture}[Almost Hadamard Conjecture, AHC]
The only complex almost Hadamard matrices are the complex Hadamard matrices.
\end{conjecture}

Let us begin our study by building on the work in sections 1-3, by examining the AHM conditions found in Theorem 3.5 above. Our claim is that a careful analysis of the inequality found there can actually lead us to a simpler statement. We first have:

\begin{proposition}
For a self-adjoint matrix $X\in M_N(\mathbb C)$, the following are equivalent:
\begin{enumerate}
\item $Tr(XA^2)\leq0$, for any anti-hermitian matrix $A\in M_N(\mathbb C)$.

\item $Tr(XB^2)\geq0$, for any hermitian matrix $B\in M_N(\mathbb C)$.

\item $Tr(XC)\geq0$, for any positive matrix $C\in M_N(\mathbb C)$.

\item $X\geq0$.
\end{enumerate}
\end{proposition}

\begin{proof}
These equivalences are well-known, the proof being as follows:

$(1)\implies(2)$ follows by taking $B=iA$. 

$(2)\implies(3)$ follows by taking $C=B^2$. 

$(3)\implies(4)$ follows by diagonalizing $X$, and then taking $C$ to be diagonal.

$(4)\implies(1)$ is clear as well, because with $Y=\sqrt{X}$ we have:
$$Tr(XA^2)=Tr(Y^2A^2)=Tr(YA^2Y)=-Tr((YA)(YA)^*)\leq0$$

Thus, the above four conditions are indeed equivalent.
\end{proof}

In view of some further discussion, let us record as well the following result:

\begin{proposition}
For a symmetric matrix $X\in M_N(\mathbb R)$, the following are equivalent:
\begin{enumerate}
\item $Tr(XA^2)\leq0$, for any antisymmetric matrix $A$.

\item The sum of the two smallest eigenvalues of $X$ is positive: $\lambda_N + \lambda_{N-1}\geq 0$, where $\lambda_N\leq\lambda_{N-1}\leq\ldots\leq \lambda_1$ are the eigenvalues of $X$. 
\end{enumerate}
\end{proposition}

\begin{proof}
Let $a ={\rm vec}(A)$ be the vectorization of $A$, given by:
$$a =\sum_{i,j=1}^NA_{ij}e_i\otimes e_j$$

Since $A$ is an antisymmetric matrix, $a$ is an antisymmetric vector, $a \in \Lambda^2(\mathbb R^N)$. It is clear (see Figure \ref{figure} below) that we have the following formula:
$$Tr(XA^2)=<X,A^2>=-<AX,A>=-<a,(I_N\otimes X)a>$$

Thus the condition (1) is equivalent to $P_-(I_N\otimes X)P_-$ being a PSD matrix, with $P_-$ being the orthogonal projector on the antisymmetric subspace in $\mathbb R^N\otimes\mathbb R^N$. 

\begin{figure}[htbp]
\centering
\includegraphics[bb=10 1 200 75,width=0.5\textwidth]{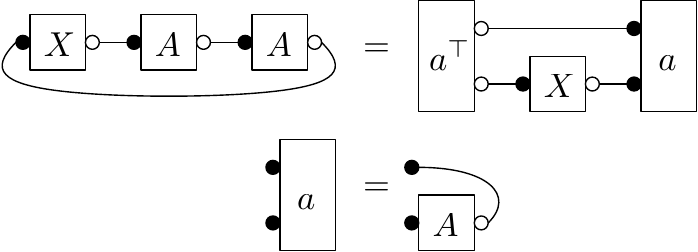}
\caption{From antisymmetric matrices to antisymmetric vectors.}
\label{figure}
\end{figure}

However, for any two eigenvectors $x_i \perp x_j$ of $X$ with eigenvalues $\lambda_i, \lambda_j$, we have:
\begin{eqnarray*} 
P_-(I_N\otimes X)P_-(x_i\otimes x_j-x_j\otimes x_i) 
&=&P_-(\lambda_j x_i\otimes x_j-\lambda_i x_j\otimes x_i)\\
&=&\frac{\lambda_i +\lambda_j}{2}(x_i\otimes x_j-x_j\otimes x_i)
\end{eqnarray*}

Thus, the non-trivial eigenvalues of $P_-(I_N\otimes X)P_-$ are $(\lambda_i +\lambda_j)/2$, for every ordered pair of indices $(i<j)$, and this gives the result.
\end{proof}

We can now formulate a better result regarding the almost Hadamard matrices:

\begin{proposition}
Given $U\in U(N)$, set $S_{ij}=sgn(U_{ij})$, and $X=S^*U$.
\begin{enumerate}
\item $U$ locally maximizes the $1$-norm on $U(N)$ precisely when $X\geq0$, and when
$$\Phi(U,B)=Tr(XB^2)-\sum_{ij}\frac{Re\left[(UB)_{ij}\overline{S}_{ij}\right]^2}{|U_{ij}|}$$ 
is positive, for any hermitian matrix $B\in M_N(\mathbb C)$.

\item If $U\in O(N)$, this matrix locally maximizes the $1$-norm on $O(N)$ precisely when $X$ is self-adjoint, and the sum of its two smallest eigenvalues  is positive.
\end{enumerate}
\end{proposition}

\begin{proof}
Here (1) follows from Theorem 3.5, by setting $A=iB$, and by using Proposition 4.3, which shows that we must have indeed $X\geq0$. As for (2), this follows from (1), with the remark that the right term vanishes, and from Proposition 4.4.
\end{proof}

The result (2) above corrects an omission in our previous work \cite{bcs}, \cite{bne}, \cite{bns}, \cite{bnz}, where the stronger condition $X\geq0$ was thought to be the revelant one. However, we conjecture here that the conditions found in (2) above should actually imply $X\geq0$.

Let us study now more in detail the quantity $\Phi(U,B)$ appearing in Proposition 4.5 (1). As a first observation here, we have the following result:

\begin{proposition}
With $S_{ij}=sgn(U_{ij})$ and $X=S^*U$ as above, we have
$$\Phi(U,B)=\Phi(U,B+D)$$
for any $D\in M_N(\mathbb R)$ diagonal.
\end{proposition}

\begin{proof}
The matrices $X,B,D$ being all self-adjoint, we have $(XBD)^*=DBX$, and so when computing $\Phi(U,B+D)$, the trace term decomposes as follows:
\begin{eqnarray*}
Tr(X(B+D)^2)
&=&Tr(XB^2)+Tr(XBD)+Tr(XDB)+Tr(XD^2)\\
&=&Tr(XB^2)+Tr(XBD)+Tr(DBX)+Tr(XD^2)\\
&=&Tr(XB^2)+2Re[Tr(XBD)]+Tr(XD^2)
\end{eqnarray*}

Regarding now the second term, with $D=diag(\lambda_1,\ldots,\lambda_N)$ with $\lambda_i\in\mathbb R$ we have $(UD)_{ij}\overline{S}_{ij}=U_{ij}\lambda_j\overline{S}_{ij}=\lambda_j|U_{ij}|$, and so this term decomposes as follows:
\begin{eqnarray*}
&&\sum_{ij}\frac{Re\left[(UB+UD)_{ij}\overline{S}_{ij}\right]^2}{|U_{ij}|}\\
&=&\sum_{ij}\frac{Re\left[(UB)_{ij}\overline{S}_{ij}+\lambda_j|U_{ij}|\right]^2}{|U_{ij}|}
=\sum_{ij}\frac{\left[Re\left[(UB)_{ij}\overline{S}_{ij}\right]+\lambda_j|U_{ij}|\right]^2}{|U_{ij}|}\\
&=&\sum_{ij}\frac{Re\left[(UB)_{ij}\overline{S}_{ij}\right]^2}{|U_{ij}|}+2\sum_{ij}\lambda_jRe\left[(UB)_{ij}\overline{S}_{ij}\right]+\sum_{ij}\lambda_j^2|U_{ij}|
\end{eqnarray*}

Now observe that the middle term in this expression is given by:
\begin{eqnarray*}
2\sum_{ij}\lambda_jRe\left[(UB)_{ij}\overline{S}_{ij}\right]
&=&2Re\left[\sum_{ij}\lambda_j(UB)_{ij}\overline{S}_{ij}\right]\\
&=&2Re\left[\sum_{ij}(S^*)_{ji}(UB)_{ij}D_{jj}\right]\\
&=&2Re[Tr(XBD)]
\end{eqnarray*}

As for the term on the right in the above expression, this is given by:
\begin{eqnarray*}
\sum_{ij}\lambda_j^2|U_{ij}|=\sum_{ij}\lambda_j^2\overline{S}_{ij}U_{ij}=\sum_{ij}\overline{S}_{ij}(UD^2)_{ij}=Tr(XD^2)
\end{eqnarray*}

Thus when doing the substraction we obtain $\Phi(U,B+D)=\Phi(U,B)$, as claimed.
\end{proof}

Observe that with $B=0$ we obtain $\Phi(U,D)=0$, for any $D\in M_N(\mathbb R)$ diagonal. In other words, the inequality is Proposition 4.5 is an equality, when $B$ is diagonal.

Consider now the following matrix, which is the basic example of real AHM:
$$K_N=\frac{1}{\sqrt{N}}
\begin{pmatrix}
2-N&2&\ldots&2\\
2&2-N&\ldots&2\\
\ldots&\ldots&\ldots&\ldots\\
2&2&\ldots&2-N
\end{pmatrix}$$

We have the following result, which provides the first piece of evidence for the AHC:

\begin{theorem}
Consider the matrix $U=\frac{1}{N}(2\mathbb I_N-N1_N)$. Assuming that $B\in M_N(\mathbb R)$ is symmetric and satisfies $UB=\lambda B$, we have:
$$\Phi(U,B)=\lambda\cdot\frac{N-4}{2}\left[Tr(B^2)+\frac{\lambda N}{N-2}\sum_iB_{ii}^2\right]$$
In particular, $K_N=\sqrt{N}U$ is not complex AHM at $N\neq4$, because:
\begin{enumerate}
\item For $B=\mathbb I_N$ we have $\Phi(U,B)=\frac{N^2(N-1)(N-4)}{2(N-2)}$, which is negative at $N=3$.

\item For $B\in M_N(\mathbb R)$ nonzero, symmetric, and satisfying $B\mathbb I_N=0$, $diag(B)=0$ we have $\Phi(U,B)=(2-\frac{N}{2})Tr(B^2)$, which is negative at $N\geq5$.
\end{enumerate}
\end{theorem} 

\begin{proof}
With $U\in O(N)$, $B\in M_N(\mathbb R)$, the formula in Proposition 4.5 reads:
$$\Phi(U,B)=Tr(S^tUB^2)-\sum_{ij}\frac{(UB)_{ij}^2}{|U_{ij}|}$$

Asusming now $U=\frac{1}{N}(2\mathbb I_N-N1_N)$ and $UB=\lambda B$, this formula becomes:
$$\Phi(U,B)=\lambda\left[Tr(S^tB^2)-\lambda N\sum_{ij}\frac{B_{ij}^2}{|2-N\delta_{ij}|}\right]$$

Since we have $\mathbb I_NB=\frac{N}{2}(U+1_N)B=\frac{(\lambda+1)N}{2}B$, the trace term is:
$$Tr(S^tB^2)=Tr\left[(\mathbb I_N-21_N)B^2\right]=\left(\frac{(\lambda+1)N}{2}-2\right)Tr(B^2)$$

Regarding now the sum on the right, this can be computed as follows:
\begin{eqnarray*}
\sum_{ij}\frac{B_{ij}^2}{|2-N\delta_{ij}|}
&=&\sum_{ij}B_{ij}^2\left(\frac{1}{2}+\left(\frac{1}{N-2}-\frac{1}{2}\right)\delta_{ij}\right)\\
&=&\sum_{ij}B_{ij}^2\left(\frac{1}{2}-\frac{N-4}{2(N-2)}\delta_{ij}\right)\\
&=&\frac{1}{2}Tr(B^2)-\frac{N-4}{2(N-2)}\sum_iB_{ii}^2
\end{eqnarray*}

We obtain the following formula, which gives the one in the statement:
$$\Phi(U,B)=\lambda\left[\left(\frac{(\lambda+1)N}{2}-2-\frac{\lambda N}{2}\right)Tr(B^2)+\frac{\lambda N(N-4)}{2(N-2)}\sum_iB_{ii}^2\right]$$

We can now prove our various results, as follows:

(1) Here we have $\lambda=1$, and we obtain, as claimed:
$$\Phi(U,B)=\frac{N-4}{2}\left[N^2+\frac{N^2}{N-2}\right]=\frac{N^2(N-4)(N-1)}{2(N-2)}$$

(2) Here we have $\lambda=-1$, and we obtain, as claimed:
$$\Phi(U,B)=\left(2-\frac{N}{2}\right)Tr(B^2)$$

It remains to prove that matrices $B$ as in the statement exist, at any $N\geq5$. As a first remark, such matrices cannot exist at $N=2,3$. At $N=4$, however, we have solutions, which are as follows, with $x+y+z=0$, not all zero:
$$B=\begin{pmatrix}
0&x&y&z\\
x&0&z&y\\
y&z&0&x\\
z&y&x&0
\end{pmatrix}$$

At $N\geq5$ now, we can simply use this matrix, completed with $0$ entries.
\end{proof}

We will see later on, in section 7 below, that the above result admits a uniform proof.

\section{Block designs}

In this section and in the next ones we work out various generalizations of Theorem 4.7. As a first observation, the matrix $U=\frac{1}{N}(2\mathbb I_N-N1_N)$ appearing there has only 2 entries, $U\in M_N(x,y)$. In addition, with $a=0,b=1,c=N-2$, any two rows of this matrix look as follows, up to a permutation of the columns:
$$\begin{matrix}
x\ldots x&x\ldots x&y\ldots y&y\ldots y\\
\underbrace{x\ldots x}_a&\underbrace{y\ldots y}_b&\underbrace{x\ldots x}_b&\underbrace{y\ldots y}_c
\end{matrix}$$

Following \cite{bnz}, we call $(a,b,c)$ pattern any matrix $M\in M_N(0,1)$ having this property, with $x=0,y=1$. With this notion in hand, we have the following result:

\begin{proposition}
If $U\in M_N(x,y)$ is unitary then, up to the multiplication by a complex number of modulus $1$, one of the following must happen:
\begin{enumerate}
\item $U$ is a permutation matrix.

\item $U=H/\sqrt{N}$, with $H$ being a two-entry complex Hadamard matrix. 

\item $U$ comes from an $(a,b,c)$ pattern, by replacing the $0,1$ entries with
$$y=\frac{1}{\sqrt{(a+b)t^2+b+c}}\quad,\quad x=-\varepsilon ty$$
where $\varepsilon\in\mathbb T$ and $t>0$ are subject to the condition $at^2-2bRe(\varepsilon)t+c=0$.
\end{enumerate}
In addition, assuming that we are in the third case, and not in the second one, the transpose matrix $U^t$ comes from an $(a,b,c)$ pattern too, and $b^2-b=ac$.
\end{proposition}

\begin{proof}
Let us look at an arbitrary pair of rows of $U$. Up to a permutation of the columns, this pair of rows must look as follows:
$$\begin{matrix}
x\ldots x&x\ldots x&y\ldots y&y\ldots y\\
\underbrace{x\ldots x}_a&\underbrace{y\ldots y}_b&\underbrace{x\ldots x}_{b'}&\underbrace{y\ldots y}_c
\end{matrix}$$

The orthogonality equations for these two rows are as follows:
\begin{eqnarray*}
a|x|^2+bx\bar{y}+b'y\bar{x}+c|y|^2&=&0\\
(a+b)|x|^2+(b'+c)|y|^2&=&1\\
(a+b')|x|^2+(b+c)|y|^2&=&1
\end{eqnarray*}

Assuming $y=0$, we cannot have $x=0$, so the first equation reads $a=0$, and then the second and third equations read $b|x|^2=b'|x|^2=1$. Thus the row picture is:
$$\begin{matrix}
x\ldots x&0\ldots 0&0\ldots 0\\
\underbrace{0\ldots 0}_b&\underbrace{x\ldots x}_b&\underbrace{0\ldots 0}_c
\end{matrix}$$

Since this is true for any two rows, and we have a square matrix, we must have $b=1$, and so in this case $U$ appears as a rescaled permutation matrix, as in (1). 

Assuming now $y\neq0$, we can rescale, as to have $y>0$. Now observe that the first orthogonality equation gives $bx+b'\bar{x}\in\mathbb R$, and so $(b-b')x\in\mathbb R$, and by taking the difference between the second and third equations, we obtain $(b-b')(|x|^2-y^2)=0$. 

We therefore have two cases, as follows:

(I) Case $b\neq b'$. Here we obtain $x\in\mathbb R$, $|x|=y$, and since $x=y$ is impossible, we must have $x=-y$, and we are therefore in the Hadamard matrix case.

(II) Case $b=b'$. Here with $x=-\varepsilon ty$ with $|\varepsilon|=1$, $t>0$, the equations become:
\begin{eqnarray*}
at^2-2bRe(\varepsilon)t+c&=&0\\
((a+b)t^2+b+c)y^2&=&1
\end{eqnarray*}

Let us compute now $a,b,c$. We have three linear equations, namely the above two ones, plus the equation $a+2b+c=N$. The determinant of the corresponding system is:
$$\left|\begin{matrix}
t^2&-2tRe(\varepsilon)&1\\
t^2&t^2+1&1\\
1&2&1
\end{matrix}\right|=(t^2-1)(t+\varepsilon)(t+\bar{\varepsilon})$$

At $t=1$ we have $|x|=y$, and we are in the Hadamard matrix case. At $t\neq1$ this determinant is nonzero, so $a,b,c$ are uniquely determined by $x,y$, and we therefore have an $(a,b,c)$ pattern. In addition, the values of $x,y$ are those in the statement.

Regarding now the last assertion, our assumption that we are not in case (2) gives $t\neq1$. Thus $a,b,c$ are uniquely determined by $x,y$, and so the transpose matrix $U^t$, which is a unitary matrix with entries $x,y$, must come from an $(a,b,c)$ pattern as well.

In order to establish now the formula $b^2-b=ac$, consider the following set:
$$I_k=\left\{(i,j)\Big|U_{ij}=U_{ik}=y\right\}$$

Our claim is that, by counting this set via two different methods, and by using the fact that both $U,U^t$ come from $(a,b,c)$ patterns, we have:
$$|I_k|=\begin{cases}
(b+c)(b+c-1)\\
(N-1)c
\end{cases}$$

Indeed, there are $b+c$ choices for $i$, and then $b+c-1$ choices for $j$, and this gives the first formula. On the other hand, there are $N-1$ choices for $j$, and then $c$ choices for $i$, and this gives the second formula. Thus, we obtain the following equality:
$$b^2+c^2+2bc-b-c=ac+2bc+c^2-c$$ 

But this gives $b^2-b=ac$, and we are done.
\end{proof}

Let us impose now the critical point condition. The result here is:

\begin{theorem}
Given an $(a,b,c)$ pattern, the associated unitaries $U(x,y)$, with the normalization $y>0$, which are critical points of the $1$-norm, are as follows:
\begin{enumerate}
\item We have two real solutions, given by
$$x=-\frac{t}{\sqrt{b}(t+1)}\quad,\quad y=\frac{1}{\sqrt{b}(t+1)}$$
where $t>0$ is subject to the condition $at^2-2bt+c=0$. 

\item We have as well two complex solutions, given by
$$y=\frac{1}{\sqrt{N}}\quad,\quad x=-\frac{\varepsilon}{\sqrt{N}}$$
where $\varepsilon\in\mathbb T$ is subject to the condition $2bRe(\varepsilon)=a+c$.
\end{enumerate}
Moreover, the real and complex solutions can overlap only when $a+c=2b=\frac{N}{2}$ and when $U(-1,1)$ is an Hadamard matrix, the common solution being $U(-1,1)/\sqrt{N}$.
\end{theorem}

\begin{proof}
If we denote by $P,Q\in M_N(0,1)$ the matrices describing the positions of the $0,1$ entries inside the pattern, then we have the following formulae:
\begin{eqnarray*}
PP^t=P^tP&=&a\mathbb I_N+b1_N\\
QQ^t=Q^tQ&=&c\mathbb I_N+b1_N\\
PQ^t=P^tQ=QP^t=Q^tP&=&b\mathbb I_N-b1_N
\end{eqnarray*}

According to the formulae in Proposition 5.1 above, we have:
$$U=xP+yQ=y(Q-\varepsilon tP)$$
$$S=sgn(x)P+Q=Q-\varepsilon P$$

Thus the matrix $X=S^*U$ from the critical point criterion is given by:
\begin{eqnarray*}
S^*U
&=&y(Q^t-\bar{\varepsilon}P^t)(Q-\varepsilon tP)\\
&=&y(Q^tQ+tP^tP-\varepsilon tQ^tP-\bar{\varepsilon}P^tQ)\\
&=&t\left[c\mathbb I_N+b1_N+t(a\mathbb I_N+b1_N)-(\varepsilon t+\bar{\varepsilon})(b\mathbb I_N-b1_N)\right]\\
&=&y\left[(c+at-b(\varepsilon t+\bar{\varepsilon}))\mathbb I_N+b(1+t+\varepsilon t+\bar{\varepsilon})1_N\right]
\end{eqnarray*}

We conclude that the solutions to our problem are given by:
\begin{eqnarray*}
X=X^*
&\iff&\varepsilon t+\bar{\varepsilon}\in\mathbb R\\
&\iff&\varepsilon t+\bar{\varepsilon}=\bar{\varepsilon}t+\varepsilon\\
&\iff&\varepsilon(t-1)=\bar{\varepsilon}(t-1)\\
&\iff&\varepsilon=1\ {\rm or}\ t=1
\end{eqnarray*}

Here we have used the fact that $\varepsilon\in\mathbb T$ and  $t>0$ must satisfy $at^2-2bRe(\varepsilon)t+c=0$, which shows that we must have $Re(\varepsilon)>0$, and so that $\varepsilon\in\mathbb R$ implies $\varepsilon=1$.

By using now the formulae in Proposition 5.1, the situations $\varepsilon=1$ and $t=1$ correspond to the situations (1,2) in the statement, and we are done with the first part.

Regarding now the overlapping case, $\varepsilon=t=1$, this must come from $a+c=2b$, and so from $a+c=2b=\frac{N}{2}$. According to (2), our matrix is in this case:
$$U=\frac{1}{\sqrt{N}}(Q-P)=\frac{1}{\sqrt{N}}\cdot U(-1,1)$$

Since this matrix is unitary, $U(1,-1)$ must be Hadamard, and we are done.
\end{proof}

Let us examine now the 4 matrices found above. These matrices depend of course on the existence of $t>0$ and $\varepsilon\in\mathbb T$ as above, so we can have $0,1,2,3,4$ solutions.

Best here is to work out first an explicit, key example, as follows:

\begin{proposition}
With $a=0,b=1,c=N-2$, the critical points found above, rescaled by $\sqrt{N}$, are as follows:
\begin{enumerate}
\item The real solutions appear only at $N\geq3$, and there is only one solution for any such $N$, namely the matrix $K_N=\frac{1}{\sqrt{N}}(2\mathbb I_N-N1_N)$.

\item The complex solutions appear only at $N=2,3,4$, and here we obtain respectively $2,2,1$ complex Hadamard matrices, equivalent to $F_2,F_3,K_4$.
\end{enumerate}
In addition, the real and complex solutions overlap only at $N=4$, over the matrix $K_4$.
\end{proposition}

\begin{proof}
We use Theorem 5.2. Before rescaling, the situation is as follows:

(1) For the real solutions, $t>0$ is subject to $2t=N-2$, and so we must have $N\geq3$, and the solutions are given by $x=\frac{2-N}{N},y=\frac{2}{N}$. Thus, there is exactly one real solution at each $N\geq3$, namely the matrix $U_N=\frac{1}{N}(2\mathbb I_N-N1_N)$ from Theorem 4.7.

(2) For the complex solutions, $\varepsilon\in\mathbb T$ is subject to $2Re(\varepsilon)=N-2$, and so we must have $N=2,3,4$. These cases give respectively $Re(\varepsilon)=0,\frac{1}{2},1$, and so the matrices are as follows, with $w$ being one of the two nontrivial solutions of $w^3=1$:
$$\pm\frac{1}{\sqrt{2}}\begin{pmatrix}i&1\\1&i\end{pmatrix}\quad,\quad
\frac{1}{\sqrt{3}}\begin{pmatrix}w&1&1\\1&w&1\\1&1&w\end{pmatrix}\quad,\quad
\frac{1}{2}\begin{pmatrix}-1&1&1&1\\1&-1&1&1\\1&1&-1&1\\1&1&1&-1\end{pmatrix}$$

Now when rescaling everything by $\sqrt{N}$, and doing some elementary Hadamard equivalence manipulations on the matrices found above, we obtain the result. 

Finally, the last assertion is clear, by comparing the solutions in (1) and (2).
\end{proof}

The matrices in Theorem 5.2 (2) being rescaled complex Hadamard matrices, we will exclude them from our study. For more on such matrices, see \cite{szo}. 

Regarding the matrices in Theorem 5.2 (1), there are many interesting examples here, coming from the symmetric balanced incomplete block designs (BIBD). We have:

\begin{proposition}
Assume that $(X,B)$ is a symmetric BIBD, in the sense that $B\subset\mathcal P(X)$ with $|B|=|X|=a+2b+c$ has the following properties:
\begin{enumerate}
\item The elements (blocks) of $B$ have the same size, $a+b$.

\item Each pair of distinct points of $X$ is contained in exactly $a$ blocks of $B$.
\end{enumerate}
The corresponding adjacency matrix $M_{iB}=\delta_{i\in B}$ is then an $(a,b,c)$ pattern. 
\end{proposition}

\begin{proof}
This follows indeed from the basic theory of symmetric BIBD, see \cite{bnz}.
\end{proof}

As a basic example here, consider the Fano plane, pictured below. The 7 points and the 7 lines form a symmetric BIBD, with parameters $(a,b,c)=(1,2,2)$:
\begin{figure}[htbp]
\centering
\includegraphics[bb=10 20 250 225,width=0.3\textwidth]{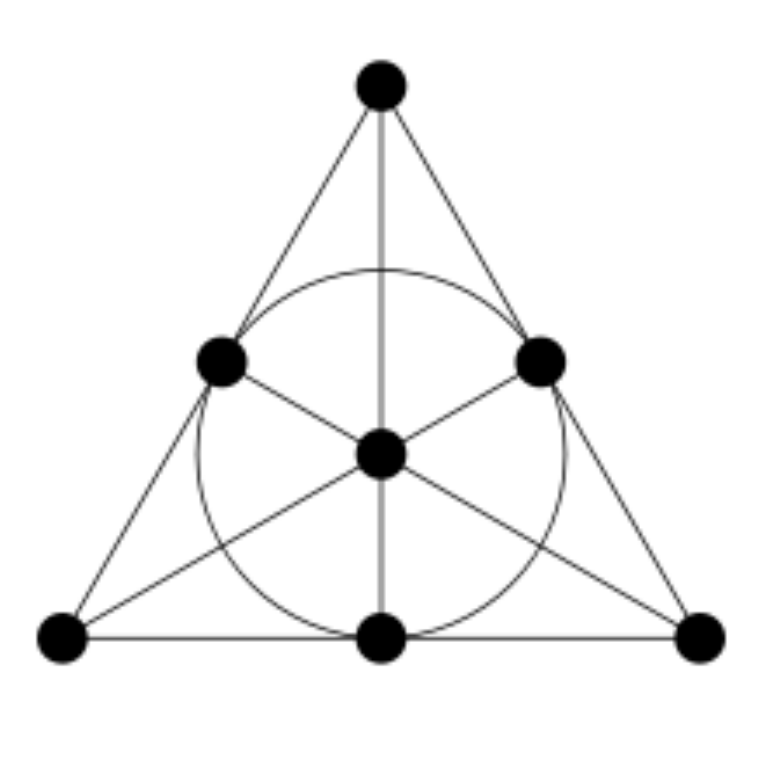}
\caption{The Fano plane.}
\label{figg}
\end{figure}

The corresponding adjacency matrix is then a $(1,2,2)$ pattern, and with $x=2-4\sqrt{2}$, $y=2+3\sqrt{2}$, coming from Theorem 5.2 (1), we obtain a real AHM. See \cite{bnz}.

Now recall that the Fano plane is the projective plane over $\mathbb F_2=\{0,1\}$. The same method works with $\mathbb F_2$ replaced by an arbitrary finite field $\mathbb F_q$, and we have:

\begin{proposition}
Assume that $q=p^k$ is a prime power. 
\begin{enumerate}
\item The incidence matrix of the projective plane over $\mathbb F_q$ is a $(1,q,q^2-q)$ pattern, and the associated matrix $U_-$, denoted $I_N^q$, with $N=q^2+q+1$, is a real AHM.

\item More generally, for any integer $d\geq1$ we have an $(a,b,c)$ pattern coming from the $d$-dimensional Grassmannian over $\mathbb F_q$, with
$$(a,b,c)=\left(\frac{q^d-1}{q-1},q^d,q^d(q-1)\right)$$
\end{enumerate}

and the associated matrix $U_-$, denoted $I_N^{q,d}$ with $N=a+2b+c$, is a real AHM.
\end{proposition}

\begin{proof}
Once again, this follows from the basic theory of symmetric BIBD, see \cite{bnz}.
\end{proof}

There are many other interesting examples of symmetric BIBD, and related real AHM. We refer to \cite{cdi}, \cite{sti} for the general theory, and to \cite{bnz} for the real AHM aspects.

\section{Exclusion results}

In this section we prove that some of the matrices found in Theorem 5.2 (1) above are not complex AHM. For this purpose, we can use the following criterion:

\begin{proposition}
Assuming $PB=\alpha B,QB=\beta B$ we have:
$$\Phi(U,B)=(\alpha x+\beta y)(x+y)\left[\frac{\beta}{x}\sum_{ij}P_{ij}B_{ij}^2-\frac{\alpha}{y}\sum_{ij}Q_{ij}B_{ij}^2\right]$$
In particular with $B=\mathbb I_N$ we obtain the formula
$$\Phi(U,\mathbb I_N)=N\lambda(a+b)(b+c)\left(\frac{y}{x}-\frac{x}{y}\right)$$
where $\lambda=(a+b)x+(b+c)y$ is the row sum of $U$.
\end{proposition}

\begin{proof}
With $U\in O(N)$, $B\in M_N(\mathbb R)$, the formula in Proposition 4.5 becomes:
$$\Phi(U,B)=Tr(S^tUB^2)-\sum_{ij}\frac{(UB)_{ij}^2}{|U_{ij}|}$$

Our assumptions $PB=\alpha B,QB=\beta B$ show successively that we have:
$$UB=(xP+yQ)B=(\alpha x+\beta y)B$$
$$S^tUB=(Q^t-P^t)UB=(\alpha x+\beta y)(\beta-\alpha)B$$

The trace term in the above formula is therefore given by:
\begin{eqnarray*}
Tr(S^tUB^2)
&=&(\alpha x+\beta y)(\beta-\alpha)Tr(B^2)\\
&=&(\alpha x+\beta y)(\beta-\alpha)\left(\sum_{ij}P_{ij}B_{ij}^2+\sum_{ij}Q_{ij}B_{ij}^2\right)
\end{eqnarray*}

Regarding now the sum on the right, this is given by:
\begin{eqnarray*}
\sum_{ij}\frac{(UB)_{ij}^2}{|U_{ij}|}
&=&(\alpha x+\beta y)^2\sum_{ij}\frac{B_{ij}^2}{|U_{ij}|}\\
&=&(\alpha x+\beta y)^2\left(\frac{1}{y}\sum_{ij}Q_{ij}B_{ij}^2-\frac{1}{x}\sum_{ij}P_{ij}B_{ij}^2\right)
\end{eqnarray*}

By summing, we obtain the formula in the statement. Finally, with $B=\mathbb I_N$ we have $\alpha=a+b,\beta=b+c$, and we obtain:
$$\Phi(U,B)=((a+b)x+(b+c)y)(x+y)\left[\frac{b+c}{x}\cdot N(a+b)-\frac{a+b}{y}\cdot N(b+c)\right]$$

But this gives the formula in the statement, and we are done.
\end{proof}

As we will see later on, the above criterion excludes some of the known real AHM, but has its limitations. Thus, we are in need of more exclusion criteria. 

When $P,Q$ are symmetric, we have as well some extra directions $B$, coming from:

\begin{proposition}
Assuming $P=P^t,Q=Q^t$, the solutions $B=uP+vQ+w1_N$ of $PB=\alpha B$, $QB=\beta B$ are, up to a multiplication by a scalar, as follows:
\begin{enumerate}
\item $B=\mathbb I_N$. Here $\alpha=a+b$, $\beta=b+c$.

\item $B=1-U_-$. Here $\alpha=\sqrt{b}$, $\beta=-\sqrt{b}$.

\item $B=1+U_+$. Here $\alpha=-\sqrt{b}$, $\beta=\sqrt{b}$.
\end{enumerate}
\end{proposition}

\begin{proof}
Let us first solve the equation $PB=\alpha B$. We have:
\begin{eqnarray*}
PB
&=&uP^2+vPQ+wP\\
&=&u(a\mathbb I_N+b1_N)+v(b\mathbb I_N-b1_N)+wP\\
&=&(ua+vb)\mathbb I_N+(u-v)b1_N+wP\\
&=&(ua+vb+w)P+(ua+vb)Q+(u-v)b1_N
\end{eqnarray*}

We conclude that we have the following equivalences:
\begin{eqnarray*}
PB=\alpha B
&\iff&[ua+vb+w=\alpha u\ ,\ ua+vb=\alpha v\ ,\ ub-vb=\alpha w]\\
&\iff&[w=\alpha(u-v)\ ,\ ua+vb=\alpha v\ ,\ (u-v)b=\alpha^2(u-v)]
\end{eqnarray*}

In the case $u=v$ we obtain $w=0$, and so $B=u\mathbb I_N$. In the case $u\neq v$ we must have $\alpha=\pm\sqrt{b}$, and with the choice $v=\mp1$, the solutions are:
\begin{eqnarray*}
\alpha=\sqrt{b}&,&u=\frac{b-\sqrt{b}}{a}\quad,\quad v=-1,\quad,\quad w=\sqrt{b}(u-v)\\
\alpha=-\sqrt{b}&,&u=-\frac{b+\sqrt{b}}{a}\quad,\quad v=1,\quad,\quad w=\sqrt{b}(v-u)
\end{eqnarray*}

In order to further process these extra solutions, we use the equation $at^2-2bt+c=0$. In terms of the solutions $t=t_\pm$, the above two extra solutions are:
\begin{eqnarray*}
B_-&=&t_-P-Q+\sqrt{b}(t_-+1)=-x_-P-y_-Q+1=1-U_-\\
B_+&=&-t_+P+Q+\sqrt{b}(t_++1)=x_+P+y_+Q+1=1+U_+
\end{eqnarray*}

Let us solve now the equation $QB=\beta B$. Here we have:
\begin{eqnarray*}
QB
&=&uQP+vQ^2+wQ\\
&=&u(b\mathbb I_N-b1_N)+v(c\mathbb I_N+b1_N)+wQ\\
&=&(ub+vc)\mathbb I_N+(v-u)b1_N+wQ\\
&=&(ub+vc)P+(ub+vc+w)Q+(v-u)b1_N
\end{eqnarray*}

We conclude that we have the following equivalences:
\begin{eqnarray*}
QB=\beta B
&\iff&[ub+vc+w=\beta v\ ,\ ub+vc=\beta u\ ,\ (v-u)b=\beta w]\\
&\iff&[w=\beta(v-u)\ ,\ ub+vc=\beta u\ ,\ (v-u)b=\beta^2(v-u)]
\end{eqnarray*}

In the case $u=v$ we obtain $w=0$, and so $B=u\mathbb I_N$. In the case $u\neq v$ we must have $\beta=\pm\sqrt{b}$, and with the choice $v=\pm1$, the solutions are:
\begin{eqnarray*}
\beta=\sqrt{b}&,&u=-\frac{c}{b-\sqrt{b}}\quad,\quad v=1,\quad,\quad w=\sqrt{b}(v-u)\\
\beta=-\sqrt{b}&,&u=\frac{c}{b+\sqrt{b}}\quad,\quad v=-1,\quad,\quad w=\sqrt{b}(u-v)
\end{eqnarray*}

Now by using $b^2-b=ac$ we conclude that these two solutions coincide with those found above, for the equation $PB=\alpha B$, and this finishes the proof.
\end{proof}

Let us go back to the matrices $U_\pm$ from Theorem 5.2 (1). As it is known from \cite{bnz}, and explained in Proposition 6.4 (4) below, $U_+$ is not AHM. Thus we are only interested in $U_-$, and for dealing with it, we will only need the direction $B=1-U_-$. We have:

\begin{proposition}
Assuming $P=P^t,Q=Q^t$ and $b^2-b=ac$, with $U=U_-$ we have
$$\Phi(U,1-U)=b(y^2-x^2)\left[N(\lambda-2)+Tr(\widetilde{U})\right]$$
where $\lambda=(a+b)x+(b+c)y$ is the row sum of $U$, and $\widetilde{U}=\frac{P}{x}+\frac{Q}{y}$.
\end{proposition}

\begin{proof}
We use the general formula of $\Phi(U,B)$ found in Proposition 6.1. The multiplication parameters $\alpha,\beta$ being in our case $\alpha=\sqrt{b},\beta=-\sqrt{b}$, we obtain:
\begin{eqnarray*}
\Phi(U,B)
&=&(\sqrt{b}x-\sqrt{b}y)(x+y)\left[-\frac{\sqrt{b}}{x}\sum_{ij}P_{ij}B_{ij}^2-\frac{\sqrt{b}}{y}\sum_{ij}Q_{ij}B_{ij}^2\right]\\
&=&b(y^2-x^2)\left[\frac{1}{x}\sum_{ij}P_{ij}B_{ij}^2+\frac{1}{y}\sum_{ij}Q_{ij}B_{ij}^2\right]\\
&=&b(y^2-x^2)\left[\frac{1}{x}\sum_{ij}P_{ij}(\delta_{ij}-x)^2+\frac{1}{y}\sum_{ij}Q_{ij}(\delta_{ij}-y)^2\right]
\end{eqnarray*}

By expanding the quantities on the right, we obtain:
\begin{eqnarray*}
\Phi(U,B)
&=&b(y^2-x^2)\left[x\sum_{ij}P_{ij}+\left(\frac{1}{x}-2\right)\sum_iP_{ii}+y\sum_{ij}Q_{ij}+\left(\frac{1}{y}-2\right)\sum_{ij}Q_{ii}\right]\\
&=&b(y^2-x^2)\left[x\sum_{ij}P_{ij}+y\sum_{ij}Q_{ij}-2Tr(P+Q)+Tr\left(\frac{P}{x}+\frac{Q}{y}\right)\right]\\
&=&b(y^2-x^2)\left[N(a+b)x+N(b+c)y-2N+Tr\left(\frac{P}{x}+\frac{Q}{y}\right)\right]\\
&=&b(y^2-x^2)\left[N(\lambda-2)+Tr\left(\frac{P}{x}+\frac{Q}{y}\right)\right]
\end{eqnarray*}

Thus we have reached to the formula in the statement, and we are done.
\end{proof}

In order to apply our criteria, we use the following result:

\begin{proposition}
The matrices $U_\pm$ from Theorem 5.2 (1) are as follows:
\begin{enumerate}
\item The row sum for $U_\pm$ is $\mp1$.

\item We have $|x|>|y|$ for $U_-$ precisely when $a<b<\frac{a+c}{2}$.

\item We have $|y|>|x|$ for $U_+$ precisely when $b<a,\frac{a+c}{2}$.

\item We have $X>0$ for $U_\pm$ precisely when $0\geq\pm(c-a)$.
\end{enumerate}
\end{proposition}

\begin{proof}
We use the above results, along with some previous computations from \cite{bnz}:

(1) The row sum is indeed given by the following formula:
$$\lambda=\frac{(b+c)-(a+b)t}{\sqrt{b}(t+1)}
=\frac{(ab+ac)-(a+b)(b\pm\sqrt{b})}{\sqrt{b}(a+b\pm\sqrt{b})}
=\mp1$$

(2) For the matrix $U_-$, this follows from:
$$|x|>|y|\iff t>1\iff\sqrt{b^2-ac}<b-a$$

(3) For the matrix $U_+$ the study is similar, as follows:
$$|y|>|x|\iff t<1\iff\sqrt{b^2-ac}<a-b$$

(4) This follows indeed by using the formula of $X=S^*U$ from the proof of Theorem 5.2, and we refer to \cite{bnz} for details.
\end{proof}

We are now in position of stating our main result, regarding the $(a,b,c)$ patterns, or rather the main examples of such patterns, coming from \cite{bnz}:

\begin{theorem}
The following matrices are not complex AHM:
\begin{enumerate}
\item The matrices $K_N=\frac{1}{N}(2\mathbb I_N-N1_N)$, with $N\neq4$.

\item The matrices $I_N^{q,d}$ coming from the Grassmannians over $\mathbb F_q$.

\item The matrix $P_{11}$ coming from the Paley biplane.
\end{enumerate}
\end{theorem}

\begin{proof}
Here (1) is from Theorem 4.7. Regarding (2,3), the idea is that the matrices here are excluded either by Proposition 6.1, or by Proposition 6.3. In order to study $I_N^{q,d}$, we can use the criterion in Proposition 6.4 (2), and we obtain:
$$|x|\geq|y|\iff\frac{q^d-1}{q-1}<q^d<\frac{\frac{q^d-1}{q-1}+q^d(q-1)}{2}\iff q>2$$

Thus the matrices $I_N^{2,d}$ are excluded by Proposition 6.1. In general now, with $q=r^2$, the parameters $(a,b,c)$ and the smallest root of $at^2-2bt+c=0$ are given by:
$$a=\frac{r^{2d}-1}{r^2-1}\quad,\quad b=r^{2d}\quad,\quad c=r^{2d}(r^2-1)\quad,\quad t=\frac{r^d(r^2-1)}{r^d+1}$$

Let us compute now the quantities appearing in Proposition 6.3. Since the diagonal of the design has the same structure as the rows and the columns, we have:
\begin{eqnarray*}
Tr(\widetilde{U})
&=&\frac{a+b}{x}+\frac{b+c}{y}\\
&=&\sqrt{b}\big[(b+c)(t+1)-(a+b)(t^{-1}+1)\big]\\
&=&r^d\left[r^{2d+2}\cdot\frac{r^{d+2}+1}{r^d+1}-\frac{r^{2d+2}-1}{r^2-1}\cdot\frac{r^{d+2}+1}{r^d(r^2-1)}\right]\\
&=&(r^{d+2}+1)\left[\frac{r^{3d+2}}{r^d+1}-\frac{r^{2d+2}-1}{(r^2-1)^2}\right]
\end{eqnarray*}

On the other hand, since we have $\lambda=1$ by Proposition 6.4 (1), we obtain:
\begin{eqnarray*}
N(\lambda-2)
&=&-N\\
&=&-\big[(a+b)+(b+c)\big]\\
&=&-\left[\frac{r^{2d+2}-1}{r^2-1}+r^{2d+2}\right]\\
&=&-\frac{r^{2d+4}-1}{r^2-1}
\end{eqnarray*}

Now by summing the above two quantities, we obtain the following formula:
$$N(\lambda-2)+Tr(\widetilde{U})=(r^{d+2}+1)\left[\frac{r^{3d+2}}{r^d+1}-\frac{r^{2d+2}-1}{(r^2-1)^2}-\frac{r^{2d+4}-1}{(r^2-1)(r^{d+2}+1)}\right]$$

Now by getting back to the integer $q=\sqrt{r}$, and performing a numeric study, we conclude that this quantity is positive precisely for $q>2$. Thus, we have $\Phi(U,1-U)<0$ at $q>2$, and so the matrices $I_N^{q,d}$ with $q>2$ are excluded by Proposition 6.3.

Finally, the Paley biplane matrix $P_{11}$ is excluded by Proposition 6.1.
\end{proof}

As a comment, what happens for $K_4,I_N^{q,d},P_{11}$ is that $\Phi(U,\mathbb I_N)\Phi(U,1-U)<0$. So, our conjecture would be that, under suitable assumptions, this inequality should hold. The problem, however, is that $K_N$ with $N\geq5$ is not covered by this conjecture. We will see later on that these matrices are best approached with a random derivative method.

\section{Circulant matrices}

We recall that a matrix $U\in M_N(\mathbb C)$ is called circulant if we have $U_{ij}=\gamma_{j-i}$, for a certain vector $\gamma\in\mathbb C^N$. In this section we study the circulant AHM.

We fix $N\in\mathbb N$, and we denote by $F\in U(N)$ the rescaled Fourier matrix, $F=\frac{1}{\sqrt{N}}(w^{ij})_{ij}$ with $w=e^{2\pi i/N}$. The following result, already used in \cite{bne}, \cite{bnz}, is well-known:

\begin{proposition}
A matrix $U\in M_N(\mathbb C)$ is circulant, $U_{ij}=\gamma_{j-i}$ with $\gamma\in\mathbb C^N$, if and only if it is Fourier-diagonal, $U=FQF^*$ with $Q\in M_N(\mathbb C)$ diagonal. If so is the case, then with $Q=diag(q_0,\ldots,q_{N-1})$ we have $\gamma=\frac{1}{\sqrt{N}}F^*q$, and the following happen:
\begin{enumerate}
\item $U$ is unitary precisely when $q\in\mathbb T^N$.

\item $U$ is self-adjoint precisely when $q\in\mathbb R^N$.

\item $U$ is real precisely when $\bar{q}_i=q_{-i}$, for any $i$.
\end{enumerate}
\end{proposition}

\begin{proof}
Assuming $U_{ij}=\gamma_{j-i}$, the matrix $Q=F^*UF$ is indeed diagonal, given by:
$$Q_{ij}=\frac{1}{N}\sum_{kl}w^{jl-ik}\gamma_{l-k}=\delta_{ij}\sum_rw^{jr}\gamma_r$$ 

Conversely, if $Q\in M_N(\mathbb C)$ is diagonal then $U=FQF^*$ is circulant, given by:
$$U_{ij}=\sum_kF_{ik}Q_{kk}\bar{F}_{jk}=\frac{1}{N}\sum_kw^{(i-j)k}Q_{kk}$$

Thus we have the equivalence in the statement, and the connecting formula $\gamma=\frac{1}{\sqrt{N}}F^*q$ is clear as well, from the above formula of $U_{ij}$. Regarding now the other assertions:

(1) This is clear from $U=FQF^*$, because $Q\in U(N)\iff q\in\mathbb T^N$.

(2) By using the formula $\gamma=\frac{1}{\sqrt{N}}F^*q$, we obtain:
$$\bar{\gamma}_i=\gamma_{-i}\iff\sum_jw^{ij}\bar{q}_j=\sum_jw^{ij}q_j\iff\sum_jw^{ij}(\bar{q}_j-q_j)=0$$

This system admits the unique solution $\bar{q}_j-q_j=0$, and the result follows.

(3) We use the same method as above. From $\gamma=\frac{1}{\sqrt{N}}F^*q$ we obtain:
$$\bar{\gamma}_i=\gamma_i\iff\sum_jw^{ij}\bar{q}_j=\sum_jw^{-ij}q_j\iff\sum_jw^{ij}(\bar{q}_j-q_{-j})=0$$

This system admits the unique solution $\bar{q}_j-q_{-j}=0$, and the result follows.
\end{proof}

With this Fourier analysis picture in hand, let us study now the circulant AHM. We first have the following result, which basically goes back to \cite{bne}, \cite{bnz}:

\begin{proposition}
Given a circulant matrix $U\in M_N(\mathbb C)^*$, the matrix $X=S^*U$ is circulant too, and we have $X=FLF^*$, with $U\to L$ being obtained as follows:
\begin{enumerate}
\item We write $U_{ij}=\gamma_{j-i}$, and set $\varepsilon={\rm sgn}(\gamma)$.

\item We construct the vector $\rho_i=\sum_r\bar{\varepsilon}_r\gamma_{i+r}$.

\item We set $\lambda=\sqrt{N}\cdot F\rho$, and then $L=diag(\lambda_0,\ldots,\lambda_{N-1})$
\end{enumerate}
\end{proposition}

\begin{proof}
In terms of the vectors $\varepsilon$ and $\rho$ constructed in the statement, the matrix $X=S^*U$ is given by the following formula:
$$X_{ij}=\sum_k\bar{S}_{ki}U_{kj}=\sum_k\bar{\varepsilon}_{i-k}\gamma_{j-k}=\sum_r\bar{\varepsilon}_r\gamma_{j-i+r}=\rho_{j-i}$$

Thus $X$ is circulant, with $\rho$ as first row vector, and by using Proposition 7.1 above, we conclude that we have $X=FLF^*$, with $L$ being the matrix in the statement.
\end{proof}

In general, the verification of the critical point condition $X=X^*$ is a quite tricky question. However, as observed in \cite{bne}, \cite{bnz}, this condition is automatic when $U$ is orthogonal and symmetric. In the unitary and self-adjoint case the same holds, because we have:
$$\bar{\rho}_i=\sum_r\varepsilon_r\bar{\gamma}_{i+r}=\sum_r\bar{\varepsilon}_{-r}\gamma_{-i-r}=\rho_{-i}$$

Note that this follows as well from Proposition 2.6 (2) above.

Let us restrict now attention to the orthogonal and symmetric case. Here there are several interesting examples, known since \cite{bne}, \cite{bnz}, and with the verification of $X>0$ already done. We will show now that these matrices are not complex AHM.

As a first exclusion criterion for such matrices, we can use:

\begin{proposition}
For a circulant matrix $U\in O(N)$, $U_{ij}=\gamma_{j-i}$, we have
$$\Phi(U,\mathbb I_N)=Nu(Ns-uw)$$
where $u,s,v$ are the row sums of $U,S$ and $W_{ij}=\frac{1}{|U_{ij}|}$. Thus $\Phi(U,\mathbb I_N)<0$ when
$$\mathbb E(sgn(\gamma_i))<\mathbb E(\gamma_i)\mathbb E\left(\frac{1}{|\gamma_i|}\right)$$ 
where the symbol $\mathbb E$ stands for ``average''.
\end{proposition}

\begin{proof}
We have $U\mathbb I_N=u\mathbb I_N$, which gives the following formulae:
$$Tr(S^tU\mathbb I_N^2)=NTr(S^tU\mathbb I_N)=NuTr(S^t\mathbb I_N)=N^2us$$
$$\sum_{ij}\frac{(U\mathbb I_N)_{ij}^2}{|U_{ij}|}=u^2\sum_{ij}\frac{1}{|U_{ij}|}=Nu^2w$$

By substracting, we obtain the formula in the statement. 
\end{proof}

Here is another exclusion criterion, which is useful as well:

\begin{proposition}
If $U\in U(N)$ is circulant, $U_{ij} = \gamma_{i-j}$, and self-adjoint, we have:
$$\Phi(U,U) = N\left( -\frac{1}{|\gamma_0|} + \sum_i |\gamma_i| \right)$$
\end{proposition}

\begin{proof}
Since $U$ is circulant and hermitian, we have $U = F\operatorname{diag}(q) F^*$, for some vector $q \in \{\pm 1\}^N$. The first term in the expression of $\Phi(U,U)$ reads:
$$Tr[S^*U\cdot U^2] = Tr[S^*U] = \sum_{ij} |U_{ij}| = N \sum_i |\gamma_i|$$

For the second term in $\Phi$, note that $(U^2)_{ij} = \delta_ij$ and then:
$$S_{ii} = \operatorname{sign}(U_{ii}) = \operatorname{sign}(\gamma_0) = \operatorname{sign}\left(\sum_i q_i\right) \in \{\pm 1\}$$

We therefore obtain:
$$\sum_{ij} \frac{Re[(U^2)_{ij}\bar S_{ij}]^2}{|U_{ij}|} = \sum_i \frac{1}{|\gamma_0|} = \frac{N}{|\gamma_0|}$$

But this finishes the proof.
\end{proof}

Here is now a more advanced result, making use of a random derivative method:

\begin{proposition}
If $U\in U(N)$ is circulant, $U_{ij}=\gamma_{j-i}$, and self-adjoint, we have
$$\mathbb E(\Phi(U,B))=N\sum_i|\gamma_i|-\frac{1}{2}\left(\frac{1}{|\gamma_0|}+\frac{1-e}{|\gamma_{N/2}|}+\sum_i\frac{1}{|\gamma_i|}\right)$$
where $e=0,1$ is the parity of $N$ and $\mathbb E$ denotes the expectation with respect to the uniform measure on the set of circulant self-adjoint unitary matrices $B$.
\end{proposition}

\begin{proof}
Since $B$ is circulant, we diagonalize it as $B = F diag(\beta_i) F^*$. From Proposition 7.1, the requirement that $B$ is unitary and self-adjoint amounts to $\beta_i = \pm 1$. The expectation is taken in the probability space where the random variables $\beta_i$ are i.i.d., with symmetric Bernoulli distributions $(\delta_{-1} + \delta_1)/2$; in particular, we have $\mathbb E[\beta_i \beta_j] = \delta_{ij}$. 

Using $B^2=1_N$, the first term in the expression of $\Phi(U,B)$ reads:
$$Tr(S^*UB^2)=Tr(S^*U)=\sum_{ij}|U_{ij}|=N\sum_i|\gamma_i|$$

For the second term in the formula of $\Phi$, we develop first:
$$Re[(UB)_{ij}\bar{S}_{ij}]^2=\frac{1}{4}\left[(UB)_{ij}^2\bar{S}_{ij}^2+ \overline{(UB)}_{ij}^2S_{ij}^2+2(UB)_{ij}\overline{(UB)}_{ij}\right]$$

We then have the following computation:
\begin{eqnarray*}
\mathbb E(UB)_{ij}^2 
&=&\mathbb E (F diag(q) diag(\beta) F^*)^2_{ij}\\
&=&N^{-2}\sum_{kl}w^{(k+l)(i-j)}q_kq_l\mathbb E(\beta_k\beta_l)\\
&=&N^{-2}\sum_{kl}w^{(k+l)(i-j)}q_kq_l\delta_{kl}\\
&=&N^{-2}\sum_kw^{2k(i-j)} 
\end{eqnarray*}

We therefore obtain the following formula:
$$\mathbb E(UB)_{ij}^2=\begin{cases}
N^{-1}&\quad\ {\rm if}\ 2(i-j)=0\ (mod\ N)\\
0  &\quad {\rm otherwise}
\end{cases}$$

Similarly, we have the following formula:
$$\mathbb E(UB)_{ij}\overline{(UB)}_{ij}=N^{-2}\sum_{kl}w^{(k-l)(i-j)}q_k \bar q_l  \mathbb E(\beta_k\beta_l)= N^{-2} \sum_k |q_k|^2= N^{-1}$$

Since in both the cases $i=j$ and $i=j+N/2$ (when $N$ is even), we have $S_{ij}\in\{\pm 1\}$, the above two formulae are all that we need, and we obtain the following formula:
$$\mathbb E\left[Re[(UB)_{ij}\bar S_{ij}]^2\right]=\frac{1}{4}\left[ 2N^{-1} \delta_{ij} + 2(1-e)N^{-1}\delta_{i,j+N/2} + 2N^{-1} \right]$$

Now by summing over $i,j$, and then taking into account as well the first term in the expression of $\Phi(U,B)$, computed above, we obtain the formula in the statement. 
\end{proof}

In the orthogonal case now, we have a similar result, as follows:

\begin{theorem}
If $U\in O(N)$ is circulant, $U_{ij}=\gamma_{j-i}$, and symmetric, we have
$$\mathbb E(\Phi(U,B))=N\sum_i |\gamma_i| - \left( \frac{1}{|\gamma_0|} +\frac{1-e}{|\gamma_{N/2}|} + \frac{N-2+e}{N}\sum_i \frac{1}{|\gamma_i|}\right)$$
where $e=0,1$ is the parity of $N$ and $\mathbb E$ denotes the expectation with respect to the uniform measure on the set of circulant symmetric orthogonal matrices $B$.
\end{theorem}

\begin{proof}
As before, the expectation is taken with respect to the distribution of the eigenvalues $\beta_0, \ldots,\beta_{N-1} = \pm 1$ of $B$, which are constrained in this case by the extra condition $\beta_i = \beta_{i-i}$. The first term in the expression of $\Phi(U,B)$ is equal to $N \sum_i |\gamma_i|$. For the second term in $\Phi$, we need the following covariance term: 
$$\mathbb E(\beta_k\beta_l)=\begin{cases}
1 &\quad \text{ if } k \pm l = 0\\
0 &\quad \text{ otherwise}
\end{cases}$$

Since all the quantities are real in this case, we have (recall that $q_k = q_{-k} = \pm 1$):
\begin{eqnarray*}
\mathbb E (UB)_{ij}^2 &=& N^{-2}\sum_{kl}w^{(k+l)(i-j)}q_kq_l \mathbb E(\beta_k\beta_l)\\
&=& N^{-2}\sum_{kl}w^{(k+l)(i-j)}q_kq_l(\delta_{k,l} + \delta_{k, -l} - \delta_{2k, 2l, 0}) \\
&=& N^{-2}\left[ \sum_{k}w^{2k(i-j)}q_k^2 + \sum_k q_k q_{-k} - q_0^2 - (1-e)q_{N/2}^2\right]  \\
&=& N^{-2}\left[ N \delta_{2i,2j}  +N - 2+e \right] 
\end{eqnarray*}

We have then:
$$\sum_{ij} N^{-1} |U_{ij}|^{-1} \delta_{2i,2j} = \sum_k |\gamma_{k}|^{-1} \delta_{2k,0} = \frac{1}{|\gamma_0|} + \frac{1-e}{|\gamma_{N/2}|}$$

We have as well:
$$\sum_{ij} N^{-2}(N-2+e) |U_{ij}|^{-1}  = \frac{N-2+e}{N} \sum_{i} \frac{1}{|\gamma_i|}$$

Putting everything together gives the formula in the statement. 
\end{proof}

As an illustration for the above methods, we can now go back to the matrices in Theorem 4.7, and find a better proof for the fact that these matrices are not complex AHM. Indeed, we have the following result, which basically solves the problem:

\begin{proposition}
With $U=\frac{1}{N}(2\mathbb I_N-N1_N)$ we have the formula
$$\mathbb E(\Phi(U,B))=\frac{4-N}{2}\left(N-4-\frac{2+e}{N-2}\right)$$
where $e=0,1$ is the parity of $N$, and where $B$ varies over the space of orthogonal circulant symmetric matrices. This quantity is $-2,0,0,-\frac{3}{2},-\frac{18}{5},\ldots$ at $N=3,4,5,6,7\ldots$
\end{proposition}

\begin{proof}
This follows indeed from the general formula in Theorem 7.6 above.
\end{proof}

We therefore recover Theorem 4.7, modulo a bit of extra work still needed at $N=5$. Regarding the case $N=5$, here the above expectation vanishes, but by using either Proposition 7.3 or Proposition 7.4, we conclude that the vanishing of the expectation must come from both positive and negative contributions, and we are done.

The above results can be used in fact for excluding all the explicit examples of circulant AHM found in \cite{bnz}. We have as well extensive computer verifications, for the AHM considered in \cite{bne}, taking their input from the computer program mentioned there. 

All these verifications suggest the following conjecture:

\begin{conjecture}
For any $U\in O(N)$ which is circulant and symmetric we have
$$\mathbb E(\Phi(U,B))\leq0$$
where $B$ varies over the space of orthogonal circulant symmetric matrices. In addition, a similar result should hold in the unitary, circulant and self-adjoint case.
\end{conjecture}

This looks like a quite subtle Fourier analytic question, that we don't know how to deal with, yet. The problem is that of exploiting the positivity of the eigenvector $L$ computed in Proposition 7.2 above, in order to obtain an upper bound for $\mathbb E(\Phi(U,B))$.

\section{Further results}

From what we have so far, Conjecture 7.8 is perhaps the most interesting statement. Now observe that if $U\in U(N)$ is circulant and self-adjoint, then  $\mathbb I_NU=U\mathbb I_N=\mathbb I_N$. In other words, such a matrix is in Sinkhorn normal form, in the sense of \cite{iwo}. As a general strategy, we believe that proving the AHC requires three ingredients, namely:
\begin{enumerate}
\item A strong Fourier analysis input (proof of Conjecture 7.8).

\item A clever extension, to the matrices in Sinkhorn normal form.

\item A final extension, using some tricky transport maps, as in \cite{iwo}.
\end{enumerate}

Another idea would be that of looking directly for a formula of type $\mathbb E(\Phi(U,B))\leq0$, with $B$ varying over some ``simple'' manifold associated to $U$, but this is probably quite naive. As an illustration here, here a ``rough'' computation, valid for any $U$:

\begin{proposition}
For any $U\in U(N)^*$, we have the formula
$$\mathbb E(\Phi(U,G+G^*))=(2N-1)\sum_{ij}|U_{ij}|-\sum_{ij}|U_{ij}|^{-1}$$
with $G$ being a random matrix having i.i.d. standard complex Gaussian entries.
\end{proposition}

\begin{proof}
Regarding the trace term in the formula of $\Phi(U,B)$, we have:
\begin{eqnarray*}
\mathbb E\left[Tr(XB^2)\right] 
&=&Tr[X(\mathbb E G^2 + \mathbb E (G^*)^2 +\mathbb E GG^* +\mathbb E G^*G)]\\
&=&2N Tr(X)
=2N \sum_{ij} |U_{ij}|
\end{eqnarray*}

Regarding now the sum in the formula of $\Phi(U,B)$, we have:
\begin{eqnarray*}
\mathbb E Re\left[(UB)_{ij}\overline{S}_{ij}\right]^2 
&=&\frac{1}{4}(\mathbb E(UB)_{ij}^2 \bar S_{ij}^2 +\mathbb E\overline{(UB)_{ij}}^2 S_{ij}^2 + 2\mathbb E(UB)_{ij} \overline{(UB)_{ij}})\\
&=&\frac{1}{4}(2U_{ij}^2 \bar S_{ij}^2+2\bar U_{ij}^2S_{ij}^2+2)
=1 + |U_{ij}|^2
\end{eqnarray*}

Now by summing, this gives the formula in the statement.
\end{proof}

Let us go back now to the inequality in Proposition 4.5. When $U$ is a rescaled complex Hadamard matrix we have of course equality, and in addition, the following happens:

\begin{proposition}
For a rescaled complex Hadamard matrix, a stronger version of the inequality in Proposition 4.5 holds, with the real part replaced by the absolute value.
\end{proposition}

\begin{proof}
Indeed, for a rescaled Hadamard matrix $U=H/\sqrt{N}$ we have $S=H=\sqrt{N}U$, and thus $X=\sqrt{N}1_N$. We therefore obtain: 
\begin{eqnarray*}
\Phi(U,B) 
&=&\sqrt{N}\left[Tr(B^2)-\sum_{ij}Re\left[(UB)_{ij}\overline{S}_{ij}\right]^2\right]\\
&\geq&\sqrt{N}\left[Tr(B^2)-\sum_{ij}|(UB)_{ij}\overline{S}_{ij}|^2\right]\\
&=&\sqrt{N}\left[Tr(B^2)-\sum_{ij}|(UB)_{ij}|^2\right]\\
&=&\sqrt{N}\left[Tr(B^2)-Tr(UB^2U^*)\right]\\
&=&0
\end{eqnarray*}

But this proves our claim, and we are done.
\end{proof}

We have the following result, in relation with the notion of defect, from \cite{tz2}:

\begin{theorem}
For a rescaled complex Hadamard matrix, the space
$$E_U=\left\{B\in M_N(\mathbb C)\Big|B=B^*,\Phi(U,B)=0\right\}$$
is isomorphic, via $B\to[(UB)_{ij}\overline{U}_{ij}]_{ij}$, to the following space:
$$D_U=\left\{A \in M_N(\mathbb R)\Big|\sum_k\bar{U}_{ki}U_{kj}(A_{ki}-A_{kj})=0,\forall i,j\right\}$$
In particular the two ``defects'' $\dim_\mathbb RE_U$ and $\dim_\mathbb RD_U$ coincide.
\end{theorem}

\begin{proof}
Since a self-adjoint matrix $B\in M_N(\mathbb C)$ belongs to $E_U$ precisely when the only inequality in the proof of Proposition 8.2 above is saturated, we have:
$$E_U=\left\{B\in M_N(\mathbb C)\Big|B=B^*,Im\left[(UB)_{ij}\overline{U}_{ij}\right]=0,\forall i,j\right\}$$

The condition on the right tells us that the matrix $A=(UB)_{ij}\bar{U}_{ij}$ must be real. Now since the construction $B\to A$ is injective, we obtain an isomorphism, as follows:
$$E_U\simeq\left\{A\in M_N(\mathbb R)\Big|A_{ij}=(UB)_{ij}\bar{U}_{ij}\implies B=B^*\right\}$$

Our claim is that the space on the right is $D_U$. Indeed, let us pick $A\in M_N(\mathbb R)$. The condition $A_{ij}=(UB)_{ij}\bar{U}_{ij}$ is then equivalent to $(UB)_{ij}=NU_{ij}A_{ij}$, and so in terms of the matrix $C_{ij}=U_{ij}A_{ij}$ we have $(UB)_{ij}=NC_{ij}$, and so $UB=NC$. Thus $B=NU^*C$, and we can now perform the study of the condition $B=B^*$, as follows:
\begin{eqnarray*}
B=B^*
&\iff&U^*C=C^*U\\
&\iff&\sum_k\bar{U}_{ki}C_{kj}=\sum_k\bar{C}_{ki}U_{kj},\forall i,j\\
&\iff&\sum_k\bar{U}_{ki}U_{kj}A_{kj}=\sum_k\bar{U}_{ki}A_{ki}U_{kj},\forall i,j
\end{eqnarray*}

Thus we have reached to the condition defining $D_U$, and we are done. 
\end{proof}

Finally, we have the following conjecture:

\begin{conjecture}
For a matrix $U\in\sqrt{N}U(N)$, the following are equivalent:
\begin{enumerate}
\item $U$ is a strict AHM.

\item $U$ is an isolated CHM.

\item $U$ is a CHM with minimal defect.
\end{enumerate}
\end{conjecture}

Here $(3)\implies(2)\implies(1)$ both hold. Indeed, $(3)\implies(2)$ is clear, and $(2)\implies(1)$ follows from the fact that the CHM are the unique global maximizers of the 1-norm.

Regarding now $(1)\implies(2)\implies(3)$, observe that $(1)\implies(2)$ would follow from the AHC. As for $(2)\implies(3)$, this is a well-known conjecture.

As a conclusion, assuming that the AHC holds, the conjecture ``isolated CHM implies minimal CHM defect'' is equivalent to the conjecture ``strict AHM implies minimal AHM defect''. Thus, we would have here an AHM approach to a CHM question. As explained in the introduction, there are of course many other potential applications of the AHC.


\begin{thebibliography}{99}

\bibitem{bcs}T. Banica, B. Collins and J.-M. Schlenker, On orthogonal matrices maximizing the 1-norm, {\em Indiana Univ. Math. J.} {\bf 59} (2010), 839--856.

\bibitem{bne}T. Banica and I. Nechita, Almost Hadamard matrices: the case of arbitrary exponents, {\em Discrete Appl. Math.} {\bf 161} (2013), 2367--2379.

\bibitem{bns}T. Banica, I. Nechita and J.-M. Schlenker, Submatrices of Hadamard matrices: complementation results, {\em Electron. J. Linear Algebra} {\bf 27} (2014), 197--212.

\bibitem{bnz}T. Banica, I. Nechita and K. \.Zyczkowski, Almost Hadamard matrices: general theory and examples, {\em Open Syst. Inf. Dyn.} {\bf 19} (2012), 1--26.

\bibitem{bbe}N. Barros e S\'a and I. Bengtsson, Families of complex Hadamard matrices, {\em Linear Algebra Appl.} {\bf 438} (2013), 2929--2957.

\bibitem{bni}K. Beauchamp and R. Nicoara, Orthogonal maximal abelian $*$-subalgebras of the $6\times6$ matrices, {\em Linear Algebra Appl.} {\bf 428} (2008), 1833--1853.

\bibitem{bjo}G. Bj\"orck, Functions of modulus $1$ on ${\rm Z}_n$ whose Fourier transforms have constant modulus, and cyclic $n$-roots, {\em NATO Adv. Sci. Inst. Ser. C Math. Phys. Sci.} {\bf 315} (1990), 131--140.

\bibitem{but}A.T. Butson, Generalized Hadamard matrices, {\em Proc. Amer. Math. Soc.} {\bf 13} (1962), 894--898.

\bibitem{cdi}C.J. Colbourn and J.H. Dinitz, Handbook of combinatorial designs, CRC Press (2007).

\bibitem{ha1}U. Haagerup, Orthogonal maximal abelian $*$-subalgebras of the $n\times n$ matrices and cyclic $n$-roots, in ``Operator algebras and quantum field theory'', International Press (1997), 296--323.

\bibitem{ha2}U. Haagerup, Cyclic $p$-roots of prime lengths $p$ and related complex Hadamard matrices, {\tt arxiv: 0803.2629}.

\bibitem{had}J. Hadamard, R\'esolution d'une question relative aux d\'eterminants, {\em Bull. Sci. Math.} {\bf 2} (1893), 240--246.

\bibitem{hor}K.J. Horadam, Hadamard matrices and their applications, Princeton Univ. Press (2007).

\bibitem{iwo}M. Idel and M.M. Wolf, Sinkhorn normal form for unitary matrices, {\em Linear Algebra Appl.} {\bf 471} (2015), 76--84.

\bibitem{kar}B.R. Karlsson, Three-parameter complex Hadamard matrices of order 6, {\em Linear Algebra Appl.} {\bf 434} (2011), 247--258. 

\bibitem{kta}H. Kharaghani and B. Tayfeh-Rezaie, A Hadamard matrix of order 428, {\em J. Combin. Des.} {\bf 13} (2005), 435--440.

\bibitem{lle}W. de Launey and D.A. Levin, A Fourier-analytic approach to counting partial Hadamard matrices, {\em Cryptogr. Commun.}  {\bf 2} (2010), 307--334. 

\bibitem{lsc}K.H. Leung and B. Schmidt, New restrictions on possible orders of circulant Hadamard matrices, {\em Des. Codes Cryptogr.} {\bf 64} (2012), 143--151.

\bibitem{rys}H.J. Ryser, Combinatorial mathematics, Wiley (1963).

\bibitem{sti}D.R. Stinson, Combinatorial designs: constructions and analysis, Springer-Verlag (2006).

\bibitem{syl}J.J. Sylvester, Thoughts on inverse orthogonal matrices, simultaneous sign-successions, and tesselated pavements in two or more colours, with applications to Newton's rule, ornamental tile-work, and the theory of numbers, {\em Phil. Mag.} {\bf 34} (1867), 461--475.

\bibitem{szo}F. Sz\"{o}ll\H{o}si, Exotic complex Hadamard matrices and their equivalence, {\em Cryptogr. Commun.} {\bf 2} (2010), 187--198. 

\bibitem{tz1}W. Tadej and K. \.Zyczkowski, A concise guide to complex Hadamard matrices, {\em Open Syst. Inf. Dyn.} {\bf 13} (2006), 133--177.

\bibitem{tz2}W. Tadej and K. \.Zyczkowski, Defect of a unitary matrix, {\em Linear Algebra Appl.} {\bf 429} (2008), 447--481.

\end{thebibliography}
\end{document}